\documentclass[12pt,letterpaper,showkeys]{article} 
\usepackage{graphicx}
\usepackage{amssymb}
\usepackage{amsmath}
\usepackage{mathrsfs}
\graphicspath{{figures/}}
\newtheorem{theorem}{Theorem}[section] 
\newtheorem{lemma}[theorem]{Lemma} 
\newtheorem{corollary}[theorem]{Corollary} 
\newtheorem{proposition}[theorem]{Proposition} 
\newtheorem{remark}[theorem]{Remark}
 
\newtheorem{definition}{Definition}
\newcommand{\beq}{\begin{equation}} 
\newcommand{\eeq}{\end{equation}} 
\newcommand{\beqa}{\begin{eqnarray}} 
\newcommand{\eeqa}{\end{eqnarray}} 
\newcommand{\beqas}{\begin{eqnarray*}} 
\newcommand{\eeqas}{\end{eqnarray*}} 
\newcommand{\ba}{\begin{array}} 
\newcommand{\ea}{\end{array}} 
\newcommand{\bi}{\begin{itemize}} 
\newcommand{\ei}{\end{itemize}} 
\newcommand{\gap}{\hspace*{2em}}

\newcommand{\sgn}{\mathop{\mathrm{sgn}}}

\def\eqnok#1{(\ref{#1})}

\def\vgap{\vspace*{.1in}}
\def\QED{\ifhmode\unskip\nobreak\fi\ifmmode\ifinner\else\hskip5pt\fi\fi
  \hbox{\hskip5pt\vrule width5pt height5pt depth1.5pt\hskip1pt}} 
\def\ae{{\alpha,\epsilon}}  
\def\Arg{{\rm Arg}}

\def\bF{{F}}

\def\cB{{\cal B}}
\def\bcB{{\bar \cB}}

\def\Diag{{\rm Diag}}
\def\eps{{\epsilon}}
\def\irl{{\rm IRL}}

\def\tx{{\tilde x}}
\def\uel{{u_{\eps}}}

\title{Iterative Reweighted Minimization Methods for $l_p$ Regularized Unconstrained 
Nonlinear Programming}  

\author{
	Zhaosong Lu% 
	\thanks{
	Department of Mathematics, Simon Fraser University, Burnaby, BC, 
	V5A 1S6, Canada. (email: {\tt zhaosong@sfu.ca}). This work was supported 
        in part by NSERC Discovery Grant.} 
	}

\date{September 28, 2012}

\begin{document}

\maketitle

\begin{abstract}

In this paper we study  general $l_p$ regularized unconstrained minimization problems.  In particular, 
we derive lower bounds for nonzero entries of first- and second-order stationary points, and hence 
also of local minimizers of the $l_p$ minimization problems. We extend some existing iterative reweighted 
$l_1$ ($\irl_1$) and $l_2$ ($\irl_2$) minimization methods to solve these problems  and proposed new variants 
for them in which each subproblem has a closed form solution. Also, we provide a unified convergence analysis 
for these methods. In addition, we propose a novel Lipschitz continuous $\eps$-approximation to $\|x\|^p_p$. 
Using this result, we develop new $\irl_1$ methods for the  $l_p$ minimization problems and showed that any 
accumulation point of the sequence generated by these methods is a first-order stationary point, provided 
that the approximation parameter $\eps$ is below a computable threshold value. This is a remarkable result 
since all existing iterative reweighted minimization methods require that $\eps$ be dynamically updated and 
approach zero. Our computational results demonstrate that the new $\irl_1$ method  is generally more stable 
than the existing $\irl_1$ methods \cite{FoLa09,ChZh10} in terms of objective function value and CPU time. 

\vskip14pt

\vskip14pt

%\noindent
%{\bf AMS 2000 subject classification:} 
\end{abstract}

\section{Introduction}

Recently numerous optimization models and methods have been proposed for finding 
sparse solutions to a system or an optimization problem (e.g., see \cite{Ti96,ChDoSa98,
CanTao05,CaRoTa06,HeGiTr06,CanTao07,Ch07,CaWaBo08,ChYi08,VaFr08,FoLa09,BlDa09,BeTe09,
HaYiZh07,WrNoFi09,YuTo11,LuZh12,ZhLi12}). 
In this paper we are interested in one of those models, namely, the $l_p$ 
regularized unconstrained nonlinear programming model 
\beq \label{lp}
\min\limits_{x\in\Re^n} \{F(x) := f(x) + \lambda \|x\|^p_p\},
\eeq 
for some $\lambda>0$ and $p \in (0,1)$, where $f$ is a smooth function 
with $L_f$-Lipschitz-continuous gradient in $\Re^n$, that is,  
\[
\|\nabla f(x)-\nabla f(y)\|_2 \le L_f \|x-y\|_2, \  \ \ \forall x, y \in \Re^n,
\] 
and $f$ is bounded below in $\Re^n$. Here, $\|x\|_p := (\sum^n_{i=1} |x_i|^p)^{1/p}$ 
for any $x\in\Re^n$. One can observe that as $p \downarrow 0$,  problem 
\eqref{lp} approaches the $l_0$ minimization problem
\beq \label{l0}
\min\limits_{x\in\Re^n} f(x) + \lambda \|x\|_0,
\eeq 
which is an exact formulation of finding a sparse vector to minimize the 
function $f$. Some efficient numerical methods such as iterative hard thresholding 
\cite{BlDa09} and penalty decomposition methods \cite{LuZh12} have recently been 
proposed for solving \eqref{l0}. In addition, as $p \uparrow 1$, problem 
\eqref{lp} approaches the $l_1$ minimization problem
\beq \label{l1}
\min\limits_{x\in\Re^n} f(x) + \lambda \|x\|_1,
\eeq 
which is a widely used convex relaxation for \eqref{l0}. When $f$ is a convex 
quadratic function, model \eqref{l1} is shown to be extremely effective in 
finding a sparse vector to minimize $f$. A variety of efficient methods 
were proposed for solving \eqref{l1} over last few years (e.g., see \cite{
VaFr08,BeTe09,HaYiZh07,WrNoFi09,YuTo11}). 
Since problem \eqref{lp} is intermediate between problems \eqref{l0} and 
\eqref{l1}, one can expect that it is also capable of seeking out a sparse 
vector to minimize $f$. As demonstrated by extensive computational studies 
in \cite{Ch07}, problem \eqref{lp} can even produce a sparser solution 
than \eqref{l1} does while both achieve similar values of $f$. 
  
A great deal of effort was recently made by many researchers (e.g., see 
\cite{Ch07,ChSt08,ChYi08,FoLa09,LaWa10,ChXuYe10,ChZh10,DaDeFoGu10,ChGeWaYe11,
GeJiYe11,Su12,BiCh12,ChNiYu12}) 
for studying problem \eqref{lp} or its related problem 
\beq \label{linear-lp} 
\min\limits_{x\in\Re^n} \{\|x\|^p_p: Ax=b\}.
\eeq  
In particular, Chartrand \cite{Ch07}, Chartrand and Staneva \cite{ChSt08}, Foucart and Lai 
\cite{FoLa09}, and Sun \cite{Su12} established some sufficient conditions for recovering the 
sparest solution to a undetermined linear system $Ax=b$ by the model \eqref{linear-lp}. 
Efficient iterative reweighted $l_1$ ($\irl_1$) and $l_2$ ($\irl_2$) minimization 
algorithms were also proposed for finding an approximate solution to \eqref{linear-lp}
by Foucart and Lai \cite{FoLa09} and Daubechies et al.\ \cite{DaDeFoGu10}, respectively. Though 
problem \eqref{linear-lp} is generally NP hard (see \cite{ChGeWaYe11,GeJiYe11}), it is shown in 
\cite{FoLa09,DaDeFoGu10} that under some assumptions, the sequence generated by $\irl_1$ and $\irl_2$ 
algorithms converges to the sparest solution to the above linear system, which is also the global 
minimizer of \eqref{linear-lp}. In addition, Chen et al.\ \cite{ChXuYe10} considered a special 
case of problem \eqref{lp} with $f(x)=\frac12\|Ax-b\|^2_2$, namely, the problem 
\beq \label{l2-lp}
\min\limits_{x\in\Re^n} \frac12\|Ax-b\|^2_2 + \lambda \|x\|^p_p.
\eeq
They derived lower bounds for nonzero entries of local minimizers of \eqref{l2-lp} 
and also proposed a hybrid orthogonal matching pursuit-smoothing gradient method 
for solving \eqref{l2-lp}. Since $\|x\|^p_p$ is non-Lipschitz continuous,  
Chen and Zhou \cite{ChZh10} recently considered the following approximation to \eqref{l2-lp}: 
\[
\min\limits_{x\in\Re^n} \frac12\|Ax-b\|^2_2 + \lambda \sum^n_{i=1}(|x_i|+\eps)^p
\]
for some small $\eps >0$. And they also proposed an $\irl_1$ algorithm to solve this approximation 
problem. Recently, Lai and Wang \cite{LaWa10} considered another approximation to \eqref{l2-lp}, 
which is
\[
\min\limits_{x\in\Re^n} \frac12\|Ax-b\|^2_2 + \lambda \sum^n_{i=1}(|x_i|^2+\eps)^{p/2},
\]
and proposed an $\irl_2$ algorithm for solving this approximation. Very recently, 
Bian and Chen \cite{BiCh12} and Chen et al.\ \cite{ChNiYu12} proposed a smoothing sequential 
quadratic programming (SQP) algorithm and a smoothing trust region Newton (TRN) method, respectively, 
for solving a class of nonsmooth nonconvex problems that include \eqref{lp} as a special case. When 
applied to problem \eqref{lp}, their methods first approximate $|x|^p_p$ by 
a suitable smooth function and then apply an SQP or a TRN algorithm to solve the resulting approximation 
problem. Lately, Bian et al. \cite{BiChYe12} proposed first- and second-order interior point algorithms 
for solving a class of non-Lipschitz and nonconvex minimization problems with {\it bounded} box constraints, 
which can be suitably applied to $l_p$ regularized minimization problems over a {\it compact} box.  

In this paper we consider general $l_p$ regularized unconstrained optimization 
problem \eqref{lp}. In particular, we first derive lower bounds for nonzero entries of first- 
and second-order stationary points, and hence also of local minimizers of \eqref{lp}. We then 
extend the aforementioned $\irl_1$ and $\irl_2$ methods \cite{FoLa09,DaDeFoGu10,LaWa10,ChZh10} 
to solve \eqref{lp} and propose some new variants for them. We also provide a unified 
convergence analysis for these methods. Finally, we propose a novel Lipschitz continuous 
$\eps$-approximation to $\|x\|^p_p$ and also propose a locally Lipschitz continuous function 
$F_{\eps}(x)$ to approximate $F(x)$. Subsequently, we develop $\irl_1$ minimization methods 
for solving the resulting approximation problem $\min_{x\in\Re^n} F_\eps(x)$. We show that 
any accumulation point of the sequence generated by these methods is a first-order 
stationary point of problem \eqref{lp}, provided that $\eps$ is below a computable threshold 
value. This is a remarkable result since all existing iterative reweighted minimization methods 
for $l_p$ minimization problems require that $\eps$ be dynamically updated and approach zero. 

The outline of this paper is as follows. In  Subsection \ref{notation} we introduce some notations 
that are used in the paper. In Section \ref{tech} we derive lower bounds for nonzero entries of 
stationary points, and hence also of local minimizers of problem \eqref{lp}. We also propose a 
locally Lipschitz continuous function $F_{\eps}(x)$ to approximate $F(x)$ and study some properties 
of the approximation problem $\min_{x\in\Re^n} F_\eps(x)$. In Section \ref{unify}, we extend the 
existing $\irl_1$ and $\irl_2$ minimization methods from problems \eqref{linear-lp} and \eqref{l2-lp} 
to general problems \eqref{lp} and  propose new variants for them. We also provide a unified convergence 
analysis for these methods. In Section \ref{new-IRL1} we propose new $\irl_1$ methods for solving \eqref{lp} 
and establish their convergence. In Section \ref{results}, we conduct numerical experiments to 
compare the performance of several $\irl_1$ minimization methods that are studied in this paper 
for \eqref{lp}. Finally, in Section \ref{conclude} we present some concluding remarks. 

\subsection{Notation} \label{notation}

The set of all $n$-dimensional positive vectors is denoted by $\Re^n_{++}$. Given any $x\in \Re^n$ 
and a scalar $\tau$, $|x|^\tau$ denotes an $n$-dimensional vector whose $i$th component is $|x_i|^\tau$. 
In addition, $\Diag(x)$ denotes an $n\times n$ diagonal matrix whose diagonal is formed by the vector $x$. 
Given an index set $\cB \subseteq \{1,\ldots, n\}$, $x_\cB$ denotes the sub-vector of $x$ indexed by $\cB$. 
Similarly, $X_{\cB\cB}$ denotes the sub-matrix of $X$ whose rows and columns are indexed by $\cB$. In addition, 
if a matrix $X$ is positive semidefinite, we write $X \succeq 0$. The sign operator is 
denoted by $\sgn$, that is, 
\[
\sgn(t) = \left\{\ba{ll}
1 & \mbox{if} \ t >0, \\
{[-1,1]} & \mbox{if} \ t=0, \\
-1 & \mbox{otherwise}. 
\ea\right.
\]
Finally,  for any $\beta <0$, we define $0^\beta = \infty$. 
     
\section{Technical results} \label{tech}

In this section we derive lower bounds for nonzero entries of stationary 
points, and hence also of local minimizers of problem \eqref{lp}. We also 
propose a nonsmooth but locally Lipschitz continuous function $F_{\eps}(x)$ 
to approximate $F(x)$. Moreover, we show that when $\eps$ is below a 
computable threshold value, a certain stationary point of the corresponding 
approximation problem $\min_{x\in\Re^n} F_\eps(x)$ is also that of 
\eqref{lp}. This result plays a crucial role  in  developing new 
$\irl_1$ methods for solving \eqref{lp} in Section \ref{new-IRL1}.

\subsection{Lower bounds for nonzero entries of stationary points of \eqref{lp}}
\label{lower-bdd}

The first- and second-order stationary points of problem \eqref{lp} are 
defined in \cite{ChXuYe10}. We first review these definitions. Then we derive 
lower bounds for nonzero entries of the stationary points, and hence also  
of local minimizers of problem \eqref{lp}. 
 
\begin{definition}
Let $x^*$ be a vector in $\Re^n$ and $X^* = \Diag(x^*)$.
$x^*\in\Re^n$ is a first-order stationary point of \eqref{lp} if 
\beq \label{1st-cond}
X^* \nabla f(x^*) + \lambda p |x^*|^p = 0.
\eeq
In addition, $x^*\in\Re^n$ is a second-order stationary 
point of \eqref{lp} if 
\beq \label{2nd-cond}
(X^*)^T \nabla^2 f(x^*) X^* + \lambda p(p-1) \Diag(|x^*|^p) \ \succeq \ 0.
\eeq 
\end{definition}

Similar to general unconstrained smooth optimization, we can show that 
any local minimizer of \eqref{lp} is also a stationary point that is 
defined above.

\begin{proposition} \label{nec-conds}
Let $x^*$ be a local minimizer of \eqref{lp} and $X^* = \Diag(x^*)$. 
The following statements hold:
\bi
\item[(i)] $x^*$ is a first-order stationary point, that is, \eqref{1st-cond} 
holds at $x^*$. 
\item[(ii)] Further, if $f$ is twice continuously differentiable in a neighborhood 
of $x^*$, then $x^*$ is a second-order stationary point, that is, \eqref{2nd-cond} 
holds at $x^*$. 
\ei
\end{proposition}

\begin{proof}
(i) Let $\cB = \{i: x^*_i \neq 0\}$. Since $x^*$ is a local minimizer of 
\eqref{lp}, one can observe that $x^*$ is also a local minimizer of 
\beq \label{sub-lp}
\min\limits_{x\in\Re^n} \{f(x) + \lambda \|x_\cB\|^p_p: x_i = 0, \ i \notin \cB\}.
\eeq
Note that the objective function of \eqref{sub-lp} is differentiable at $x^*$. 
The first-order optimality condition of \eqref{sub-lp} yields 
\[
\frac{\partial f(x^*)}{\partial x_i} + \lambda p |x^*_i|^{p-1} \sgn(x^*_i) = 0,  
\ \forall i \in \cB.
\] 
Multiplying by $x^*_i$ both sides of the above equality, we have 
\[
x^*_i\frac{\partial f(x^*)}{\partial x_i} + \lambda p |x^*_i|^{p}=0, \ \forall i\in\cB.
\] 
Since $x^*_i=0$ for $i\notin \cB$, we observe that the above equality also holds for 
$i\notin \cB$. Hence, \eqref{1st-cond} holds.

(ii) By the assumption, we observe that the objective function of \eqref{sub-lp} 
is twice continuously differentiable at $x^*$. The second-order optimality 
condition of \eqref{sub-lp} yields 
\[
\nabla^2_{\cB\cB}f(x^*) +  \lambda p(p-1) \Diag(|x^*_\cB|^{p-2}) \ \succeq \ 0,
\]
which, together with the fact that $X^*=\Diag(x^*)$ and $x^*_i=0$ for $i \notin \cB$, 
implies that \eqref{2nd-cond} holds.
\end{proof}

\gap

Recently, Chen et al.\ \cite{ChXuYe10} derived some interesting lower bounds 
for the nonzero entries of local minimizers of problem \eqref{lp} for the special 
case where $f(x)=\frac12 \|Ax-b\|^2$ for some $A\in\Re^{m \times n}$ and $b\in\Re^m$.  
We next establish similar lower bounds for the nonzero entries of stationary points, 
and hence also of local minimizers of general problem \eqref{lp}.

\begin{theorem} \label{2nd-lower-bdd}
Let $x^*$ be a second-order stationary point of \eqref{lp} and $\cB=\{i:x^*_i \neq 0\}$. 
Suppose that $f$ is twice continuously differentiable in a neighborhood of $x^*$. Then 
the following statement holds:
\beq \label{2nd-bdd}
|x^*_i| \ \ge \ \left(\frac{\lambda p(1-p)}{L_f}\right)^{\frac{1}{2-p}}, \ \ \ \forall i 
\in \cB.
\eeq
 \end{theorem}
 
\begin{proof}
Since $f$ is twice continuously differentiable in a neighborhood of $x^*$ and 
$f$ has $L_f$-Lipschitz-continuous gradient in $\Re^n$, we see that 
$\|\nabla^2 f(x^*)\|_2 \le L_f$. In addition, since $x^*$ satisfies \eqref{2nd-cond}, we have
\[
e^T_i [(X^*)^T \nabla^2 f(x^*) X^*] e_i + \lambda p(p-1) e^T_i\Diag(|x^*|^{p-2})] e_i \ge 0, 
\] 
where $e_i$ is the $i$th coordinate vector. It then follows that for each $i \in \cB$, 
\[
[\nabla^2 f(x^*)]_{ii} + \lambda p(p-1) |x^*_i|^{p-2} \ge 0,
\]
which yields
\[
|x^*_i| \ge \left(\frac{\lambda p(1-p)}{[\nabla^2 f(x^*)]_{ii}}\right)^{\frac{1}{2-p}} \ge 
\left(\frac{\lambda p(1-p)}{\|\nabla^2 f(x^*)\|_2}\right)^{\frac{1}{2-p}} \ \ge \ 
\left(\frac{\lambda p(1-p)}{L_f}\right)^{\frac{1}{2-p}},  \ \ \ \forall i \in \cB.
\] 
\end{proof}
 
\gap

\begin{theorem} \label{1st-lower-bdd}
Let $x^*$ be a first-order stationary point satisfying $F(x^*) \le F(x^0)+\eps$ for 
some $x^0 \in\Re^n$ and $\eps \ge 0$, and let $\underline{f} = \inf_{x\in\Re^n} f(x)$ 
and $\cB = \{i:x^*_i \neq 0\}$. Then the following statement holds:
\beq \label{1st-bdd}
|x^*_i| \ \ge \ \left(\frac{\lambda p}{\sqrt{2L_f[F(x^0)+\eps-\underline{f}]}}
\right)^{\frac{1}{1-p}}, 
\ \ \  \forall i \in \cB.
\eeq
\end{theorem}
 
\begin{proof}
Since $f$ has $L_f$-Lipschitz-continuous gradient in $\Re^n$, it is well known that
\[
f(y) \le f(x) + \nabla f(x)^T(y-x) + \frac{L_f}{2}\|y-x\|^2_2, \ \ \ \forall x,y\in \Re^n.
\]
Letting $x=x^*$ and $y=x^*-\nabla f(x^*)/L_f$, we obtain that 
\beq \label{grad-bdd1}
f(x^*-\nabla f(x^*)/L_f) \ \le \ f(x^*) - \frac{1}{2L_f} \|\nabla f(x^*)\|^2_2.
\eeq
Note that 
\[
f(x^*-\nabla f(x^*)/L_f) \ge \inf_{x\in\Re^n} f(x) = \underline{f}, \ \ \ \ \  
f(x^*) \le F(x^*) \le F(x^0) + \eps.
\]
Using these relations and \eqref{grad-bdd1}, we have  
\beq \label{grad-bdd2}
\|\nabla f(x^*)\|_2 \ \le \ \sqrt{2L_f[f(x^*)-f(x^*-\nabla f(x^*)/L_f)]} \ \le \ 
\sqrt{2L_f[F(x^0)+\eps-\underline{f}]}.
\eeq
Since $x^*$ satisfies \eqref{1st-cond}, we obtain that for every $i\in \cB$,
\[
|x^*_i| = \left(\frac{1}{\lambda p}\left|\frac{\partial f(x^*)}{\partial x_i}\right|\right)^{\frac{1}{p-1}} 
\ \ge \ \left(\frac{\|\nabla f(x^*)\|_2}{\lambda p}\right)^{\frac{1}{p-1}},
\]
which together with \eqref{grad-bdd2} yields
\[
|x^*_i| \ \ge \ \left(\frac{\lambda p}{\sqrt{2L_f[F(x^0)+\eps-\underline{f}]}}\right)^{\frac{1}{1-p}}, 
\ \ \ \forall i \in \cB.
\]
\end{proof}
 
\subsection{Locally Lipschitz continuous approximation to \eqref{lp}}
\label{lip-approx}

It is known that for $p\in (0,1)$, the function $\|x\|_p^p$ is not locally 
Lipschitz continuous at some points in $\Re^n$ and the Clarke subdifferential 
does not exist there (see, for example, \cite{ChXuYe10}). This brings a great deal of 
challenge for designing algorithms for solving problem \eqref{lp}. In this 
subsection we propose a nonsmooth but Lipschitz continuous $\eps$-approximation to 
$\|x\|_p^p$ for every $\eps >0$. As a consequence, we obtain a nonsmooth but locally 
Lipschitz continuous $\eps$-approximation $F_{\eps}(x)$ to $F(x)$. Furthermore, we 
show that when $\eps$ is below a computable threshold value, a certain stationary point 
of the corresponding approximation problem $\min_{x\in\Re^n} F_\eps(x)$ is also 
that of \eqref{lp}.      
 
\begin{lemma} \label{approx-abs}
Let $u >0$ be arbitrarily given, and let $q$ be such that 
\beq \label{q}
\frac{1}{p}+\frac{1}{q}=1.
\eeq
Define 
\beq \label{h}
h_u(t) := \min\limits_{0 \le s \le u} p\left(|t|s-\frac{s^q}{q}\right), \ \ \ \forall t \in \Re.
\eeq
Then the following statements hold:
\bi
\item[(i)]
 $0 \ \le \ h_u(t)-|t|^p \ \le \ u^q$ for every $t \in \Re$.
 \item[(ii)] $h_u$ is $p u$-Lipschitz continuous in $(-\infty,\infty)$, 
i.e., 
\[
|h_u(t_1) - h_u(t_2)| \ \le p u |t_1-t_2|, \ \ \  \forall t_1, t_2\in\Re.
\] 
\item[(iii)] The Clarke subdifferential of $h_u$, denoted by $\partial h_u$, exists 
everywhere, and it is given by 
\beq \label{deriv}
\partial h_u(t) = \left\{\ba{ll}
p|t|^{p-1} \sgn(t) & \ \mbox{if} \  |t| > u^{q-1}, \\ [6pt]
p u \sgn(t) & \ \mbox{if} \ |t| \le u^{q-1}.
\ea\right.
\eeq 
\ei
\end{lemma} 

\begin{proof}
(i) Let $g_t(s) = p(|t|s-s^q/q)$ for $s>0$. Since $p\in (0,1)$, we observe from \eqref{q} 
that $q<0$. It then implies that $g_t(s) \to \infty$ as $s \downarrow 0$. This together with 
the continuity of $g_t$ implies that $h_u(t)$ is well-defined for all $t\in\Re$. In addition, 
it is easy to show that $g_t(\cdot)$ is convex in $(0,\infty)$, and moreover,
$\inf\limits_{s>0} g_t(s) \ = \ |t|^p$. Hence, we have
\[
h_u(t) \ = \ \min\limits_{0 \le s \le u} g_t(s) \ \ge \ \inf\limits_{s>0} g_t(s) \ = \ 
|t|^p, \ \ \ \forall t \in \Re.
\]
We next show that 
$\ h_u(t)-|t|^p \ \le \ u^q$ by dividing its proof into two cases.
\bi
\item[1)] Assume that $|t| > u^{q-1}$. Then, the optimal value of \eqref{h} is 
achieved at $s^* = |t|^{\frac{1}{q-1}}$ and hence, 
\[
h_u(t) \ = \ p\left(|t|s^*-\frac{(s^*)^q}{q}\right) \ = \ |t|^p. 
\] 
\item[2)] Assume that $|t| \le u^{q-1}$. It can be shown that the 
optimal value of \eqref{h} is achieved at $s^*=u$. Using this result and 
the relation $|t| \le u^{q-1}$, we obtain that 
\[
h_u(t) \ = \ p\left(|t|u-\frac{u^q}{q}\right) \ \le \ 
p\left( u^{q-1} u-\frac{u^q}{q}\right) \ = \ u^q,
\]  
which implies that $h_u(t) - |t|^p \ \le \ h_u(t) \ \le \ u^q$. 
\ei
Combining the above two cases, we conclude that statement (i) holds. 

\vgap

(ii) Let $\phi: [0,\infty) \to \Re$ be defined as follows: 
\[
\phi(t) = \left\{\ba{ll}
t^p & \mbox{if} \  t > u^{q-1}, \\ [6pt]
p(t u-u^q/q) & \mbox{if} \   0 \le t \le u^{q-1}.
\ea\right.
\]
It is not hard to see that 
\beq \label{phi-deriv}
\phi'(t) = \left\{\ba{ll}
pt^{p-1} & \mbox{if} \  t > u^{q-1}, \\ [6pt]
p u & \mbox{if} \   0 \le t \le u^{q-1}.
\ea\right.
\eeq
Hence, $0 \le \phi'(t) \le pu$ for every $t\in[0,\infty)$, which implies that 
$\phi$ is $p u$-Lipschitz continuous on $[0,\infty)$. In addition, one can 
observe from the proof of (i) that $h_u (t) = \phi(|t|)$ for all $t$. 
By the chain rule, we easily conclude that $h_u$ is $p u$-Lipschitz continuous 
in $(-\infty,\infty)$.  

(iii) Since $h_u$ is Lipschitz continuous everywhere, it follows from Theorem 2.5.1 of 
\cite{Clarke83} that  
\beq \label{clarke-diff}
\partial h_u(t) = {\rm cov}\left\{\lim\limits_{t_k \in D \to t} h'_u(t_k)\right\},
\eeq
where ${\rm cov}$ denotes convex hull and $D$ is the set of points at which 
$h_u$ is differentiable. Recall that $h_u(t) = \phi(|t|)$ for all $t$. Hence, 
$h'_u(t) = \phi'(|t|)\sgn(t)$ for every $t \neq 0$. Using this relation, \eqref{phi-deriv} 
and \eqref{clarke-diff}, we immediately see that statement (iii) holds.   
\end{proof} 

\gap

\begin{corollary} \label{approx-pnorm} 
Let $u >0$ be arbitrarily given, and let $h(x) = \sum^n_{i=1} h_u(x_i)$ 
for every $x\in\Re^n$, where $h_u$ is defined in \eqref{h}. Then the 
following statements hold:
\bi
\item[(i)]
 $0 \ \le \ h(x)-\|x\|^p_p \ \le \ n u^q$ for every $x\in \Re^n$.
 \item[(ii)] $h$ is $\sqrt{n}pu$-Lipschitz continuous in $\Re^n$, 
i.e., 
\[
\|h(x) - h(y)\|_2 \ \le \sqrt{n} p u \|x-y\|_2, \ \ \ \ \forall x, y.
\] 
 \item[(iii)] The Clark subdifferential of $h$ exists at every $x\in\Re^n$.
 \ei
\end{corollary}

\gap

We are now ready to propose a nonsmooth but locally Lipschitz continuous 
$\eps$-approximation to $F(x)$.     

\begin{proposition} \label{approx-lp}
Let $\eps >0$ be arbitrarily given. Define 
\beq \label{feps}
F_\eps(x) := f(x) + \lambda \sum^n_{i=1} h_{\uel}(x_i),
\eeq 
where 
\beq \label{u-eps}
h_\uel(t) := \min\limits_{0 \le s \le \uel} p\left(|t|s-\frac{s^q}{q}\right), \ \ \ \ \ \
\uel:= \left(\frac{\eps}{\lambda n}\right)^{\frac1q}.
\eeq
Then the following statements hold:
\bi
\item[(i)]
 $0 \ \le \ F_\eps(x)-F(x) \ \le \ \eps$ for every $x\in\Re^n$. 
\item[(ii)] $F_\eps$ is locally Lipschitz continuous in $\Re^n$. Furthermore, if $f$ is 
Lipschitz continuous, so is $F_\eps$.
 \item[(iii)] The Clark subdifferential of $F_\eps$ exists at every $x\in\Re^n$.
 \ei
\end{proposition}

\begin{proof}
Using the definitions of $F_\eps$ and $F$, we 
have 
$
F_\eps(x)-F(x) = \lambda (\sum^n_{i=1} h_{\uel}(x_i) - \|x\|^p_p), 
$
which, together with Corollary \ref{approx-pnorm} (i) with $u=\uel$, implies that statement (i) 
holds. Since $f$ is differentiable in $\Re^n$, it is known that $f$ is locally Lipschitz continuous. In 
addition, we know from Corollary \ref{approx-pnorm} (ii) that $\sum^n_{i=1} h_{\uel}(x_i)$ is 
Lipschitz continuous in $\Re^n$. These facts imply that statement (ii) holds. Statement (iii) 
immediately follows from  Corollary \ref{approx-pnorm} (iii).
\end{proof}

\gap

From Proposition \ref{approx-lp}, we know that $F_\eps$ is a nice $\eps$-approximation to $F$. 
It is very natural to find an approximate solution of \eqref{lp} by solving the corresponding 
$\eps$-approximation problem
\beq \label{lp-eps}
\min\limits_{x\in\Re^n} F_\eps(x), 
\eeq
where $F_\eps$ is defined in \eqref{feps}. Strikingly, we can show that when $\eps$ 
is below a computable threshold value, a certain stationary point of problem \eqref{lp-eps} is also 
that of \eqref{lp}.   
 
\begin{theorem} \label{stat-pts}
Let $x^0 \in \Re^n$ be an arbitrary point, and let $\eps$ be such that 
 \beq \label{eps-uppbdd}
0 < \eps < n\lambda \left[\frac{\sqrt{2L_f[F(x^0)+\eps-\underline{f}]}}{\lambda p}\right]^q,
\eeq 
where $\underline{f} = \inf_{x\in\Re^n} f(x)$. Suppose that $x^*$ is a first-order stationary 
point of \eqref{lp-eps} such that that $F_\eps(x^*) \le F_\eps(x^0)$. 
Then, $x^*$ is also a first-order stationary point of \eqref{lp}, i.e., \eqref{1st-cond} holds at 
$x^*$. Moreover, the nonzero entries of $x^*$ satisfy the first-order lower bound 
\eqref{1st-bdd}.
\end{theorem}

\begin{proof}
Let $\cB=\{i:x^*_i \neq 0\}$. Since $x^*$ is a first-order stationary point of \eqref{lp-eps}, 
we have $0 \in \partial F_\eps(x^*)$. Hence, it follows that 
\beq \label{1st-sub-lp-eps}
\frac{\partial f(x^*)}{\partial x_i} + \lambda  \partial h_{\uel}(x^*_i) \ = \ 0,  \ \ 
\ \forall i \in \cB.
\eeq 
In addition, we notice that 
\beq \label{bdd-fxs}
f(x^*) \ \le \ F(x^*) \ \le \ F_\eps(x^*) \ \le \ F_\eps(x^0) \ \le \ F(x^0) + \eps.
\eeq
This relation together with \eqref{1st-sub-lp-eps} and \eqref{grad-bdd2} 
implies that
\beq \label{deriv-bdd}
|\partial h_{\uel}(x^*_i)| \ = \ \frac{1}{\lambda}\left|\frac{\partial f(x^*)}{\partial x_i}\right| \ \le \
\frac{1}{\lambda} \|\nabla f(x^*)\|_2 \ \le \ \frac{\sqrt{2L_f[F(x^0)+\eps-\underline{f}]}}{\lambda}, \ \ \ 
\forall i \in \cB.
\eeq
We now claim that $|x^*_i| > \uel^{q-1}$ for all $i\in\cB$, where $\uel$ is defined in \eqref{u-eps}. 
Suppose for contradiction that there exists some $i\in\cB$ such that $0 < |x^*_i| \le \uel^{q-1}$. It 
then follows from \eqref{deriv} that $|\partial h_{\uel}(x^*_i)|=p\uel$. Using this relation, 
\eqref{eps-uppbdd} and the definition of $\uel$, we obtain that
\[
|\partial h_{\uel}(x^*_i)| \ = \ p\uel = p \left(\frac{\eps}{\lambda n}\right)^{1/q}  \ > 
\ \frac{\sqrt{2L_f[F(x^0)+\eps-\underline{f}]}}{\lambda},
\]
which contradicts \eqref{deriv-bdd}. Therefore, $|x^*_i| > \uel^{q-1}$ for all $i\in\cB$. Using this 
fact and \eqref{deriv}, we see that $\partial h_{\uel}(x^*_i) \ = \ p|x^*_i|^{p-1} \sgn(x^*_i)$ for every $i\in\cB$. 
Substituting it into \eqref{1st-sub-lp-eps}, we obtain that 
\[
\frac{\partial f(x^*)}{\partial x_i} + \lambda p |x^*_i|^{p-1} \sgn(x^*_i) = 0,  \ \ 
\ \forall i \in \cB.
\]
Multiplying by $x^*_i$ both sides of the above equality, we have 
\[
x^*_i\frac{\partial f(x^*)}{\partial x_i} + \lambda p |x^*_i|^{p}=0, \ \ \ \forall i\in\cB.
\] 
Since $x^*_i=0$ for $i\notin \cB$, we observe that the above equality also holds for 
$i\notin \cB$. Hence, \eqref{1st-cond} holds. In addition, recall from \eqref{bdd-fxs} that 
$F(x^*) \ \le \ F(x^0) + \eps$. Using this relation and Theorem \ref{1st-lower-bdd}, 
we immediately see that the second part of this theorem also holds. 
\end{proof}

\gap

\begin{corollary} \label{loc-mins}
Let $x^0 \in \Re^n$ be an arbitrary point, and let $\eps$ be such that \eqref{eps-uppbdd} holds.
Suppose that $x^*$ is a local minimizer of \eqref{lp-eps} such that 
$F_\eps(x^*) \le F_\eps(x^0)$. Then the following statements hold:
\bi
\item[i)] $x^*$ is a first-order stationary point of \eqref{lp}, i.e., \eqref{1st-cond} holds at 
$x^*$. Moreover, the nonzero entries of $x^*$ satisfy the first-order lower bound 
\eqref{1st-bdd}.
\item[ii)] Suppose further that $f$ is twice continuously differentiable in a neighborhood of 
$x^*$. Then, $x^*$ is a second-order stationary point of \eqref{lp}, i.e., \eqref{2nd-cond} holds 
at $x^*$. Moreover, the nonzero entries of $x^*$ satisfy the second-order lower bound 
\eqref{2nd-bdd}.
\ei
\end{corollary}

\begin{proof}
(i) Since $x^*$ is a local minimizer of \eqref{lp-eps}, we know that $x^*$ is a stationary point 
of \eqref{lp-eps}. Statement (i) then immediately follows from Theorem \ref{stat-pts}. 

(ii) Let $\cB = \{i|x^*_i \neq 0\}$. Since $x^*$ is a local minimizer of \eqref{lp-eps}, 
we observe that $x^*$ is also a local minimizer of 
\beq \label{sub-lp-eps}
\min\limits_{x\in\Re^n} \left\{f(x) + \lambda \sum\limits_{i\in \cB} 
h_{\uel}(x_i): x_i = 0, \ i \notin \cB\right\}.
\eeq
Notice that $x^*$ is a first-order stationary point of \eqref{lp-eps}. In addition,  
$F(x^*) \ \le \ F(x^0) + \eps$ and $\eps$ satisfies \eqref{eps-uppbdd}. Using the same arguments 
as in the proof of Theorem \ref{stat-pts}, we have $|x^*_i| > \uel^{q-1}$ for all $i\in\cB$. Recall from 
the proof of Lemma \ref{approx-abs} (i) that $h_{\uel}(t) = |t|^p$ if $|t| > \uel^{q-1}$. 
Hence, $\sum\limits_{i\in \cB} h_{\uel}(x_i) = \sum\limits_{i\in \cB} |x_i|^p$ for all $x$ in 
a neighborhood of $x^*$. This, together with the fact that $x^*$ is a local minimizer of 
\eqref{sub-lp-eps}, implies that $x^*$ is also a local minimizer of \eqref{sub-lp}. The rest of 
the proof is similar to that of Proposition \ref{nec-conds} and Theorem \ref{2nd-lower-bdd}.
\end{proof}

\section{A unified analysis for some existing iterative reweighted minimization
methods} \label{unify}

Recently two types of $\irl_1$ and $\irl_2$ methods have been proposed 
in the literature \cite{FoLa09,DaDeFoGu10,LaWa10,ChZh10} for solving 
problem \eqref{linear-lp} or \eqref{l2-lp}. In this section we extend these 
methods to solve \eqref{lp} and also propose a variant of them in which each 
subproblem has a closed form solution. Moreover, we provide a unified convergence 
analysis for them. 

\subsection{The first type of $\irl_\alpha$ methods and its variant for \eqref{lp}}
\label{1st-type}

In this subsection we consider the iterative reweighted minimization methods proposed in 
\cite{LaWa10,ChZh10} for solving problem \eqref{l2-lp}, which apply an $\irl_1$ or $\irl_2$ 
method to solve a sequence of problems $\min\limits_{x\in\Re^n} Q_{1,\eps^k}(x)$ or   
$\min \limits_{x\in\Re^n}Q_{2,\eps^k}(x)$, where $\{\eps^k\}$ is a sequence of positive 
vectors approaching zero as $k \to \infty$ and 
\beq \label{Qalpha}
Q_{\ae}(x) := \frac12\|Ax-b\|^2_2 + \lambda \sum^n\limits_{i=1} 
(|x_i|^\alpha+\eps_i)^{\frac{p}{\alpha}}.
\eeq

In what follows, we extend the above methods to solve \eqref{lp} and also propose a 
variant of them in which each subproblem has a closed form solution. Moreover, we provide 
a unified convergence analysis for them. Our key observation is that problem 
\beq \label{la} 
\min\limits_{x\in\Re^n} \{\bF_{\ae}(x) := f(x) + \lambda \sum^n\limits_{i=1} 
(|x_i|^\alpha+\eps_i)^{\frac{p}{\alpha}}\}
\eeq
for $\alpha \ge 1$ and $\eps >0$ can be suitably solved by an iterative reweighted 
$l_\alpha$ ($\irl_\alpha$) method. Problem \eqref{lp} can then be solved by applying the 
$\irl_\alpha$ method to a sequence of problems \eqref{la} with $\eps=\eps^k \to 0$ as 
$k \to \infty$.

We start by presenting an $\irl_\alpha$ method for solving problem \eqref{la} as follows.

\gap

\noindent
%\begin{minipage}[h]{6.6 in}
{\bf An $\irl_\alpha$ minimization method for \eqref{la}:}  \\ [5pt]
Choose an arbitrary $x^0$. Set $k=0$. 
\begin{itemize}
\item[1)] Solve the weighted $l_\alpha$ minimization problem 
\beq \label{la-subprob}
x^{k+1} \in \Arg\min \left\{f(x)+\frac{\lambda p}{\alpha}\sum^n_{i=1} s^k_i |x_i|^\alpha\right\},
\eeq
where $s^k_i = (|x^k_i|^\alpha+\eps_i)^{\frac{p}{\alpha}-1}$ for all $i$.
\item[2)]
Set $k \leftarrow k+1$ and go to step 1). 
\end{itemize}
\noindent
{\bf end}

\vgap

We next show that any accumulation point of $\{x^k\}$ generated above is a 
first-order stationary point of \eqref{la}.

\begin{theorem} \label{converge-la}
Let the sequence $\{x^k\}$ be generated by the above $\irl_\alpha$ minimization 
method. Suppose that $x^*$ is an accumulation point of $\{x^k\}$. Then $x^*$ is 
a first-order stationary point of \eqref{la}.
\end{theorem}

\begin{proof}
Let $q$ be such that 
\beq \label{pq}
\frac{\alpha}{p} + \frac{1}{q} = 1.
\eeq
It is not hard to show that for any $\delta>0$,  
\beq \label{aux-opt}
(|t|^\alpha + \delta)^{\frac{p}{\alpha}} = \frac{p}{\alpha} \min\limits_{s \ge 0} 
\left\{(|t|^\alpha + \delta)s-\frac{s^q}{q}\right\}, \ \ \ \ \forall t \in \Re, 
\eeq 
and moreover, the minimum is achieved at $s = (|t|^\alpha + \delta)^\frac{1}{q-1}$. 
Using this result, the definition of $s^k$, and \eqref{pq}, one can observe that for $k \ge 0$,
\beq \label{abcd}
s^k = \arg\min\limits_{s \ge 0} G_\ae(x^k,s), \ \ \ 
x^{k+1} \in \Arg\min\limits_x G_\ae(x,s^k), 
\eeq
where 
\beq \label{Gxs-alpha}
G_\ae(x,s) = f(x)+ \frac{\lambda p}{\alpha} \sum^n_{i=1} \left[(|x_i|^\alpha+\eps_i)s_i- 
\frac{s^q_i}{q}\right].
\eeq
In addition, we see that $\bF_{\ae}(x^k)=G_\ae(x^k,s^k)$. It then follows that 
\beq \label{afval}
\bF_\ae(x^{k+1}) \ = \ G_\ae(x^{k+1},s^{k+1}) \ \le \ G_\ae(x^{k+1},s^k) \ \le \ G_\ae(x^k,s^k) \ = \ \bF_\ae(x^k).
\eeq
Hence, $\{\bF_\ae(x^k)\}$ is non-increasing. Since $x^*$ is an accumulation point of $\{x^k\}$, 
there exists a subsequence $K$ such that $\{x^k\}_{K} \to x^*$. By the continuity of $\bF_\ae$, 
we have $\{\bF_\ae(x^k)\}_K \to \bF_\ae(x^*)$, which together with the monotonicity of $\bF_\ae(x^k)$ 
implies that $\bF_\ae(x^k) \to \bF_\ae(x^*)$. In addition,  by the definition of $s^k$, we have $\{s^k\}_K \to s^*$, 
where $s^*_i=(|x^*_i|^\alpha+\eps_i)^{\frac{p}{\alpha}-1}$ for all $i$. 
Also, we observe that $\bF_\ae(x^*)=G_\ae(x^*,s^*)$. Using \eqref{afval} and $\bF_\ae(x^k) \to 
\bF_\ae(x^*)$, we see that $G_\ae(x^{k+1},s^k) \to \bF_\ae(x^*)=G_\ae(x^*,s^*)$. Further, 
it follows from \eqref{abcd} that 
\[
G_\ae(x,s^k) \ \ge \ G_\ae(x^{k+1},s^k) \ \ \ \forall x\in \Re^n.
\]
Upon taking limits on both sides of this inequality as $k\in K \to \infty$, we have 
\[
G_\ae(x,s^*) \ \ge \ G_\ae(x^*,s^*) \ \ \ \forall x\in \Re^n,
\]
that is, $x^* \in \Arg\min\limits_{x\in\Re^n} G_\ae(x,s^*)$, which, together with the 
first-order optimality condition and the definition of $x^*$, yields
\beq \label{la-stat-pt}
0 \in  \frac{\partial f (x^*)}{\partial x_i} + \lambda p 
(|x^*_i|^\alpha+\eps_i)^{\frac{p}{\alpha}-1} |x^*_i|^{\alpha-1} \sgn(x^*_i), \ \ \ \forall i.
\eeq
Hence, $x^*$ is a stationary point of \eqref{la}.
\end{proof}

\gap

The above $\irl_\alpha$ method needs to solve a sequence of reweighted $l_\alpha$ 
minimization problems \eqref{l1-subprob} whose solution may not be cheaply computable. 
We next propose a variant of this method in which each subproblem is much simpler and 
has a closed form solution for some $\alpha$'s (e.g., $\alpha=1$ or $2$).

\gap

\noindent
{\bf A variant of $\irl_\alpha$ minimization method for \eqref{la}:}  \\ [5pt]
Let $0< L_{\min} < L_{\max}$, $\tau>1$ and $c>0$ be given. Choose an 
arbitrary $x^0$ and set $k=0$. 
\begin{itemize}
\item[1)] Choose $L^0_k \in [L_{\min}, L_{\max}]$ arbitrarily. Set $L_k = L^0_k$.  
\bi
\item[1a)] Solve the weighted $l_\alpha$ minimization problem 
\beq \label{lalpha-close-subprob}
x^{k+1} \in \Arg\min\limits_x \left\{f(x^k)+\nabla f(x^k)^T(x-x^k) + 
\frac{L_k}{2}\|x-x^k\|^2_2 + \frac{\lambda p}{\alpha}\sum^n_{i=1} s^k_i |x_i|^\alpha\right\},
\eeq
where $s^k_i = (|x^k_i|^\alpha+\eps_i)^{\frac{p}{\alpha}-1}$ for all $i$. 
\item[1b)] If 
\beq \label{valpha-descent}
F_\ae(x^k) - F_\ae(x^{k+1}) \ge \frac{c}{2} \|x^{k+1}-x^k\|^2_2
\eeq
is satisfied, where $F_\ae$ is given in \eqref{la}, then go to step 2). 
\item[1c)] Set $L_k \leftarrow \tau L_k$ and go to step 1a).
\ei
\item[2)]
Set $k \leftarrow k+1$ and go to step 1). 
\end{itemize}
\noindent
{\bf end}

\vgap

We first show that for each outer iteration, the number of its inner iterations 
is finite.  

\begin{theorem} \label{valpha-inner}
For each $k \ge 0$, the inner termination criterion \eqref{valpha-descent} 
is satisfied after at most $\left\lceil \frac{\log(L_f+c)-\log(2L_{\min})}
{\log \tau} +2\right\rceil$ inner iterations.
\end{theorem}

\begin{proof}
Let $\bar L_k$ denote the final value of $L_k$ at the $k$th outer 
iteration. Since the objective function of \eqref{lalpha-close-subprob} is strongly convex with 
modulus $L_k$, we have 
\[
f(x^k)+\frac{\lambda p}{\alpha}\sum^n_{i=1} s^k_i |x^k_i|^\alpha\ \ge \ f(x^k)+\nabla f(x^k)^T(x^{k+1}-x^k) 
+ \frac{\lambda p}{\alpha}\sum^n_{i=1} s^k_i |x^{k+1}_i|^\alpha + L_k\|x^{k+1}-x^k\|^2_2. 
\]
Recall that $\nabla f$ is $L_f$-Lipschitz continuous. We then have 
\beq \label{lip-ineq}
f(x^{k+1}) \ \le \ f(x^k)+\nabla f(x^k)^T(x^{k+1}-x^k) + 
\frac{L_f}{2}\|x^{k+1}-x^k\|^2_2.
\eeq
Combining the above two inequalities, we obtain that 
\[
f(x^k)+\frac{\lambda p}{\alpha}\sum^n_{i=1} s^k_i |x^k_i|^\alpha \ \ge \ f(x^{k+1}) + 
\frac{\lambda p}{\alpha}\sum^n_{i=1} s^k_i |x^{k+1}_i|^\alpha  + (L_k-\frac{L_f}{2})\|x^{k+1}-x^k\|^2_2, 
\]
which together with \eqref{Gxs-alpha} yields 
\[
G_\ae(x^k,s^k) \ \ge \ G_\ae(x^{k+1},s^k) + (L_k-\frac{L_f}{2})\|x^{k+1}-x^k\|^2_2.
\]
Recall that $F_\ae(x^k)=G_\ae(x^k,s^k)$. In addition, it follows from \eqref{aux-opt} 
that $F_\ae(x) = \min\limits_{s \ge 0} G_\ae(x,s)$. Using these relations and the above 
inequality, we obtain that
\[
\ba{lcl}
F_\ae(x^{k+1}) &=& G_\ae(x^{k+1},s^{k+1}) \ \le \ G_\ae(x^{k+1},s^k) \ \le \ 
G_\ae(x^k,s^k) -(L_k-\frac{L_f}{2})\|x^{k+1}-x^k\|^2_2\\ [6pt]
&=&  F_\ae(x^k) - (L_k-\frac{L_f}{2})\|x^{k+1}-x^k\|^2_2.
\ea
\]
Hence, \eqref{valpha-descent} holds whenever $L_k \ge (L_f+c)/2$, which together with 
the definition of $\bar L_k$ implies that $\bar L_k/\tau < (L_f+c)/2$, that is, $\bar L_k <\tau(L_f+c)/2$. 
Let $n_k$ denote the number of inner iterations for the $k$th outer iteration. Then, we have 
\[
L_{\min} \tau^{n_k-1} \le L^0_k \tau^{n_k-1} =  \bar L_k <  \tau(L_f+c)/2.
\] 
Hence, $n_k \le \left\lceil \frac{\log(L_f+c)-\log(2L_{\min})}{\log \tau} +2\right\rceil$ 
and the conclusion holds.
\end{proof}

\gap

We next establish that any accumulation point of the sequence $\{x^k\}$ generated above 
is a first-order stationary point of problem \eqref{la}.

\begin{theorem} \label{valpha-outer}
Let the sequence $\{x^k\}$ be generated by the above variant of $\irl_\alpha$ method. 
Suppose that $x^*$ is an accumulation point of $\{x^k\}$. Then $x^*$ is a first-order 
stationary point of \eqref{la}.
\end{theorem}

\begin{proof}
It follows from \eqref{valpha-descent} that $\{F_\ae(x^k)\}$ is non-increasing.
Since $x^*$ is an accumulation point of $\{x^k\}$, there exists a subsequence $K$ such 
that $\{x^k\}_{K} \to x^*$. By the continuity of $F_\ae$, we have 
$\{F_\ae(x^k)\}_K \to F_\ae(x^*)$, which together with the monotonicity of $\{F_\ae(x^k)\}$ 
implies that $F_\ae(x^k) \to F_\ae(x^*)$. Using this result and \eqref{valpha-descent}, we 
can conclude that $\|x^{k+1}-x^k\| \to 0$. Let $\bar L_k$ denote the final value of $L_k$ at 
the $k$th outer iteration. From the proof of Theorem \ref{valpha-outer}, we 
know that $\bar L_k \in [L_{\min}, \tau(L_f+c)/2)$.
The first-order optimality condition of \eqref{lalpha-close-subprob} with $L_k = \bar L_k$ yields
\[
0 \in \frac{\partial f(x^k)}{\partial x_i} + \bar L_k (x^{k+1}_i-x^k_i) + \lambda p 
s^k_i |x^{k+1}_i|^{\alpha-1}\sgn(x^{k+1}_i)  = 0,  \ \ \ \forall i.
\] 
Upon taking limits on both sides of the above equality as $k\in K \to \infty$, we 
have
\[
0 \in \frac{\partial f(x^*)}{\partial x_i} + \lambda p s^*_i|x^*_i|^{\alpha-1}\sgn(x^*_i),  
\ \ \ \forall i,
\]
where $s^*_i = (|x^*_i|+\eps_i)^{\frac{p}{\alpha}-1}$ for all $i$. Hence, $x^*$ is a first-order 
stationary point of \eqref{la}.
\end{proof}

\gap

\begin{corollary} \label{approx-stat-pt-la}
Let $\delta >0$ be arbitrarily given, and let the sequence $\{x^k\}$ be generated 
by the above $\irl_\alpha$ method or its variant. Suppose that 
$\{x^k\}$ has at least one accumulation point. Then, there exists some $x^k$ such that  
\[
\|X^k \nabla f(x^k) + \lambda p |X^k|^\alpha(|x^k|^{\alpha}+\eps)^{\frac{p}{\alpha}-1}\| \ \le \ \delta,
\]
where $X^k = \Diag(x^k)$ and $|X^k|^\alpha = \Diag(|x^k|^\alpha)$.
\end{corollary}

\begin{proof}
Let $x^*$ be an arbitrary accumulation point of $\{x^k\}$. It follows from 
Theorem \ref{converge-la} that $x^*$ satisfies \eqref{la-stat-pt}. Multiplying 
by $x^*_i$ both sides of \eqref{la-stat-pt}, we have 
\[
x^*_i\frac{\partial f (x^*)}{\partial x_i} + \lambda p (|x^*_i|^\alpha
+\eps_i)^{\frac{p}{\alpha}-1} |x^*_i|^{\alpha} \ = \ 0 \ \ \ \forall i, 
\]
which, together with the continuity of $\nabla f(x)$ and $|x|^\alpha$, implies 
that the conclusion holds.
\end{proof}

\vgap

We are now ready to present the first type of $\irl_\alpha$ methods and its variant 
for solving problem \eqref{lp} in which each subproblem is in the form of \eqref{la} and solved 
by the $\irl_\alpha$ or its variant described above. The $\irl_1$ and $\irl_2$ methods 
proposed in \cite{LaWa10,ChZh10} can be viewed as the special cases of the following 
general $\irl_\alpha$ method (but not its variant) with $f(x)=\frac12\|Ax-b\|^2_2$ and 
$\alpha=1$ or $2$.

\gap

\noindent
%\begin{minipage}[h]{6.6 in}
{\bf The first type of $\irl_\alpha$ minimization methods and its variant for \eqref{lp}:}  \\ [5pt]
Let $\{\delta_k\}$ and $\{\eps^k\}$ be a sequence of positive scalars and vectors, respectively. Set $k=0$. 
\begin{itemize}
\item[1)] Apply the $\irl_\alpha$ method or its variant to problem \eqref{la} with $\eps = \eps^k$ 
for finding $x^k$  satisfying 
\beq \label{inexact-la}
\|X^k \nabla f(x^k) + \lambda p |X^k|^\alpha(|x^k|^{\alpha}+\eps^k)^{\frac{p}{\alpha}-1}\| 
\ \le \ \delta_k,
\eeq
where $X^k = \Diag(x^k)$ and $|X^k|^\alpha = \Diag(|x^k|^\alpha)$.
\item[2)]
Set $k \leftarrow k+1$ and go to step 1). 
\end{itemize}
\noindent
{\bf end}

\gap

The convergence of the above $\irl_\alpha$ method and its variant is established as follows.
 
\begin{theorem} \label{lim-stat-pt-la}
Let $\{\delta_k\}$ and $\{\eps^k\}$ be a sequence of positive scalars and vectors  
such that $\{\delta_k\} \to 0$ and $\{\eps^k\} \to 0$, respectively. Suppose that 
$\{x^k\}$ is a sequence of vectors generated above satisfying \eqref{inexact-la}, 
and that $x^*$ is an accumulation point of $\{x^k\}$. Then $x^*$ is a first-order 
stationary point of \eqref{lp}, i.e., \eqref{1st-cond} holds at $x^*$.
\end{theorem}

\begin{proof}
Let $\cB = \{i: x^*_i \neq 0\}$. It follows from \eqref{inexact-la} that 
\beq \label{inexact-xi}
\left|x^k_i \frac{\partial f (x^k)}{\partial x_i} + \lambda p |x^k_i|^\alpha(|x^k_i|^{\alpha}
+\eps^k_i)^{\frac{p}{\alpha}-1}\right| \ \le \delta_k\, \ \ \ \forall i \in \cB.
\eeq
Since $x^*$ is an accumulation point of $\{x^k\}$, there exists a subsequence $K$ such that 
$\{x^k\}_{K} \to x^*$. Upon taking limits on both sides of \eqref{inexact-xi} as $k\in K \to \infty$, 
we obtain that 
\[
x^*_i\frac{\partial f (x^*)}{\partial x_i} + \lambda p |x^*_i|^p \ = \ 0 \ \ \ \forall i\in\cB. 
\]
Since $x^*_i=0$ for $i\notin \cB$, we observe that the above equality also holds for 
$i\notin \cB$. Hence, $x^*$ satisfies \eqref{1st-cond} and it is a first-order stationary point 
of \eqref{lp}. 
\end{proof}

\subsection{The second type of $\irl_\alpha$ methods and its variant for \eqref{lp}}
\label{2nd-type}

In this subsection we are interested in the $\irl_1$ and $\irl_2$ methods proposed in 
\cite{FoLa09,DaDeFoGu10} for solving problem \eqref{linear-lp}. Given $\{\eps^k\} \subset \Re^n_{++} 
\to 0$ as $k \to \infty$, these methods solve a sequence of problems $\min\limits_{x\in\Re^n} 
Q_{1,\eps^k}(x)$ or $\min \limits_{x\in\Re^n} Q_{2,\eps^k}(x)$) extremely ``roughly'' by 
executing $\irl_1$ or $\irl_2$ method only for one iteration for each $\eps^k$, where 
$Q_{\alpha,\eps}$ is defined in \eqref{Qalpha}. 

We next extend the above methods to solve \eqref{lp} and also propose a variant of them 
in which each subproblem has a closed form solution. Moreover, we provide a unified
convergence analysis for them. We start by presenting the second type of $\irl_\alpha$ 
methods for solving \eqref{lp} as follows. They evidently become an $\irl_1$  or $\irl_2$ 
method when $\alpha=1$ or $2$.

\gap

\noindent
%\begin{minipage}[h]{6.6 in}
{\bf The second type of $\irl_\alpha$ minimization method for \eqref{lp}:}  \\ [5pt]
Let $\{\eps^k\}$ be a sequence of positive vectors in $\Re^n$. 
Choose an arbitrary $x^0$. Set $k=0$. 
\begin{itemize}
\item[1)] Solve the weighted $l_\alpha$ minimization problem 
\beq \label{la-subprob}
x^{k+1} \in \Arg\min \left\{f(x)+\frac{\lambda p}{\alpha}\sum^n_{i=1} s^k_i |x_i|^\alpha\right\},
\eeq
where $s^k_i = (|x^k_i|^\alpha+\eps^k_i)^{\frac{p}{\alpha}-1}$ for all $i$.
\item[2)]
Set $k \leftarrow k+1$ and go to step 1). 
\end{itemize}
\noindent
{\bf end}

\vgap

We next establish that any accumulation point of $\{x^k\}$ is a stationary point of \eqref{lp}.

\begin{theorem} \label{converge-la-lp}
Suppose that $\{\epsilon^k\}$ is a sequence of non-increasing positive vectors in $\Re^n$ 
and $\epsilon^k \to 0$ as $k \to \infty$. Let the sequence $\{x^k\}$ be generated by the 
above $\irl_\alpha$ method. Suppose that $x^*$ is an accumulation point of $\{x^k\}$. 
Then, $x^*$ is a stationary point of \eqref{lp}.
\end{theorem}

\begin{proof}
Let $G_\ae(\cdot,\cdot)$ be defined in \eqref{Gxs-alpha}. By the definition of $x^{k+1}$, one
 can observe that $G_{\alpha,\eps^k}(x^{k+1},s^k) \ \le 
\ G_{\alpha,\eps^k}(x^k,s^k)$. Also, by the definition of $s^{k+1}$ and a similar argument 
as in the proof of Theorem \ref{converge-la}, we have 
$G_{\alpha,\eps^{k+1}}(x^{k+1},s^{k+1}) = \inf\limits_{s \ge 0} G_{\alpha,\eps^{k+1}}(x^{k+1},s)$.
Hence, $G_{\alpha,\eps^{k+1}}(x^{k+1},s^{k+1}) \le G_{\alpha,\eps^{k+1}}(x^{k+1},s^k)$. In addiiton, since $s^k >0$ 
and $\{\epsilon^k\}$ is component-wise non-increasing, we observe that 
$G_{\alpha,\eps^{k+1}}(x^{k+1},s^k) \ \le \ G_{\alpha,\eps^k}(x^{k+1},s^k)$. Combining these inequalities, 
we have 
\beq \label{G-eps}
G_{\alpha,\eps^{k+1}}(x^{k+1},s^{k+1}) \ \le  \ G_{\alpha,\eps^{k+1}}(x^{k+1},s^k) \ \le \ G_{\alpha,\eps^k}(x^{k+1},s^k) \ \le \ G_{\alpha,\eps^k}(x^{k},s^k), 
\ \ \ \forall k \ge 0.
\eeq
Hence, $\{G_{\alpha,\eps^k}(x^{k},s^k)\}$ is non-increasing. Since $x^*$ is an accumulation 
point of $\{x^k\}$, there exists a subsequence $K$ such that $\{x^k\}_K \to x^*$. By 
the definition of $s^k$, one can verify that 
$G_{\alpha,\eps^k}(x^{k},s^k) = f(x^k)+\lambda 
\sum^n_{i=1}(|x^k_i|^\alpha+\eps^k_i)^{\frac{p}{\alpha}}$. It then follows from $\{x^k\}_K \to x^*$ 
and $\epsilon^k \to 0$ that $\{G_{\alpha,\eps^k}(x^{k},s^k)\}_K \to f(x^*)+\lambda \|x^*\|^p_p$. 
This together with the monotonicity of $\{G_{\alpha,\eps^k}(x^{k},s^k)\}$ implies that 
$G_{\alpha,\eps^k}(x^{k},s^k) \to f(x^*)+\lambda \|x^*\|^p_p$. Using this relation and 
\eqref{G-eps}, we further have 
\beq \label{G-limit}
G_{\alpha,\eps^k}(x^{k+1},s^k) \to f(x^*)+\lambda \|x^*\|^p_p.
\eeq
Let $\cB = \{i: x^*_i \neq 0\}$ and $\bcB$ be its complement in $\{1,\ldots,n\}$. We claim 
that 
\beq \label{sol-prop}
x^* \in \Arg\min\limits_{x_{\bcB}=0} \left\{f(x) + \frac{\lambda p}{\alpha}\sum\limits_{i\in\cB} 
|x_i|^\alpha |x^*_i|^{p-\alpha}\right\}.
\eeq
Indeed, using the definition of $s^k$, we see that $\{s^k_i\}_{K} \to |x^*_i|^{p-\alpha}, \ 
\forall i \in \cB$. Further, due to $s^k >0$ and $q<0$, we observe that
\[
0 \ \le \ \frac{p}{\alpha} \sum\limits_{i\in\bcB} \left[\eps^k_i s^k_i - \frac{(s^k_i)^q}{q}\right] 
\ \le \ \frac{p}{\alpha} \sum\limits_{i\in\bcB} \left[(|x^k_i|^\alpha+\eps^k_i) s^k_i - 
\frac{(s^k_i)^q}{q}\right]  \ = \ \sum\limits_{i\in\bcB} (|x^k_i|^\alpha+\eps^k_i)^{\frac{p}{\alpha}},
\]
which, together with $\eps^k \to 0$ and $\{x^k_i\}_K \to 0$ for $i\in\bcB$, implies that 
\beq \label{lim-bcB}
\lim\limits_{k\in K \to \infty} \sum\limits_{i\in\bcB} \left[\eps^k_i s^k_i - \frac{(s^k_i)^q}{q}\right] 
= 0.
\eeq
In addition, by the definition of $x^{k+1}$, we know that
$G_{\alpha,\eps^k}(x,s^k) \ \ge \ G_{\alpha,\eps^k}(x^{k+1},s^k)$. Then for every $x\in\Re^n$ such that $x_\bcB=0$, 
we have
\[
 f(x)+ \frac{\lambda p}{\alpha} \sum\limits_{i\in\cB}
\left[(|x_i|^\alpha+\eps^k_i)s^k_i- \frac{(s^k_i)^q}{q}\right] + \frac{\lambda p}{\alpha} 
\sum\limits_{i\in\bcB}
\left[\eps^k_is^k_i- \frac{(s^k_i)^q}{q}\right] \ = \ G_{\alpha,\eps^k}(x,s^k) \ \ge \ G_{\alpha,\eps^k}(x^{k+1},s^k).
\]
Upon taking limits on both sides of this inequality as $k\in K \to \infty$, and using \eqref{G-limit}, 
\eqref{lim-bcB} and the fact that $\{s^k_i\}_{K} \to  |x^*_i|^{p-\alpha}, \ 
\forall i \in \cB$, we obtain that
\[
f(x)+ \frac{\lambda p}{\alpha} \sum\limits_{i\in\cB}
\left[|x_i|^\alpha |x^*_i|^{p-\alpha} - \frac{|x^*_i|^{q(p-\alpha)}}{q}\right] 
\ \ge \  f(x^*)+\lambda \|x^*\|^p_p
\]
for all $x\in\Re^n$ such that $x_\bcB=0$. This inequality and \eqref{pq} immediately yield 
\eqref{sol-prop}. It then follows from \eqref{pq} and the first-order optimality condition of 
\eqref{sol-prop} that 
\[
 x^*_i\frac{\partial f(x^*)}{\partial x_i} + \lambda p |x^*_i|^p \ = \ 0 \ \ \ \forall i\in\cB. 
\]
Since $x^*_i=0$ for $i\in\bcB$, we observe that the above equality also holds for 
$i\in\bcB$. Hence, $x^*$ satisfies \eqref{1st-cond} and it is a stationary point 
of \eqref{lp}. 
\end{proof}

\gap

Notice that the above $\irl_\alpha$ method requires solving a sequence of reweighted 
$l_\alpha$ minimization problems \eqref{la-subprob} whose solution may not be cheaply 
computable. We next propose a variant of this method in which each subproblem is much 
simpler and has a closed form solution for some $\alpha$'s (e.g., $\alpha=1$ or $2$).

\gap

\noindent
%\begin{minipage}[h]{6.6 in}
{\bf A variant of the second type of $\irl_\alpha$ minimization method for \eqref{lp}:}  \\ [5pt]
Let $\{\eps^k\}$ be a sequence of positive vectors in $\Re^n$, and let 
$0< L_{\min} < L_{\max}$, $\tau>1$ and $c>0$ be given. Choose an arbitrary $x^0$. 
Set $k=0$. 
\begin{itemize}
\item[1)] Choose $L^0_k \in [L_{\min}, L_{\max}]$ arbitrarily. Set $L_k = L^0_k$. 
\bi
\item[1a)] Solve the weighted $l_\alpha$ minimization problem 
\beq \label{la-close-subprob}
x^{k+1} \in \Arg\min\limits_x \left\{f(x^k)+\nabla f(x^k)^T(x-x^k) + 
\frac{L_k}{2}\|x-x^k\|^2_2 + \frac{\lambda p}{\alpha}\sum^n_{i=1} s^k_i |x_i|^\alpha\right\},
\eeq
where $s^k_i = (|x^k_i|^\alpha+\eps^k_i)^{\frac{p}{\alpha}-1}$ for all $i$.
\item[1b)] If 
\beq \label{descent-gap-la}
\bF_{\ae^k}(x^k) - \bF_{\ae^{k+1}}(x^{k+1}) \ge \frac{c}{2} \|x^{k+1}-x^k\|^2_2
\eeq 
is satisfied, then go to step 2). 
\item[1c)] Set $L_k \leftarrow \tau L_k$ and go to step 1a).
\ei
\item[2)]
Set $k \leftarrow k+1$ and go to step 1). 
\end{itemize}
\noindent
{\bf end}

\vgap

We first show that for each outer iteration, the number of its inner iterations is finite.  

\begin{theorem} \label{inner-iteration-la}
For each $k \ge 0$, the inner termination criterion \eqref{descent-gap-la} 
is satisfied after at most $\left\lceil \frac{\log(L_f+c)-\log(2L_{\min})}
{\log \tau} +2\right\rceil$ inner iterations.
\end{theorem}

\begin{proof}
Let $G_\ae(\cdot,\cdot)$ be defined in \eqref{Gxs-alpha}. Since the objective function of 
\eqref{la-close-subprob} is strong convex with modulus $L_k$, we have 
\[
\ba{lcl}
f(x^k)+ \frac{\lambda p}{\alpha} \sum^n_{i=1} s^k_i |x^k_i|^{\alpha} &\ge&
\ f(x^k)+\nabla f(x^k)^T(x^{k+1}-x^k) 
+ \frac{\lambda p}{\alpha} \sum^n_{i=1} s^k_i |x^{k+1}_i|^{\alpha} + L_k\|x^{k+1}-x^k\|^2_2, \\ [6pt]
&\ge& f(x^{k+1})+ \frac{\lambda p}{\alpha}\sum^n_{i=1} s^k_i |x^{k+1}_i|^{\alpha} + 
(L_k-\frac{L_f}{2})\|x^{k+1}-x^k\|^2_2,
\ea
\]
where the last inequality is due to \eqref{lip-ineq}. This inequality together with 
the definition of $G_{\ae}$ implies that 
\[
G_{\alpha,\eps^k}(x^{k},s^k) \ \ge \ G_{\alpha,\eps^k}(x^{k+1},s^k) + (L_k-\frac{L_f}{2})\|x^{k+1}-x^k\|^2_2.
\]
In addition, by the same arguments as in the proof of Theorem \ref{converge-la-lp}, we 
have $G_{\eps^{k+1}}(x^{k+1},s^{k+1}) \ \le \ G_{\alpha,\eps^k}(x^{k+1},s^k)$. By the definitions of 
$s^k$ and $\bF_{\ae}$, one can easily verify that $G_{\alpha,\eps^k}(x^{k},s^k) = \bF_{\ae^k}(x^k)$ 
for all $k$. Combining these relations with the above inequality, we obtain that 
\[
\ba{lcl}
\bF_{\ae^{k+1}}(x^{k+1}) &=& G_{\eps^{k+1}}(x^{k+1},s^{k+1}) \ \le \ G_{\alpha,\eps^k}(x^{k+1},s^k) \ \le \ 
G_{\alpha,\eps^k}(x^{k},s^k) -(L_k-\frac{L_f}{2})\|x^{k+1}-x^k\|^2_2 \\ [6pt] 
&=&  \bF_{\ae^k}(x^k) - (L_k-\frac{L_f}{2})\|x^{k+1}-x^k\|^2_2.
\ea
\]
Hence, \eqref{descent-gap-la} holds whenever $L_k \ge (L_f+c)/2$. The rest of the proof is similar to 
that of Theorem \ref{inner-iteration}.
\end{proof}

\gap

We next show that any accumulation point of the sequence $\{x^k\}$ generated above 
is a first-order stationary point of problem \eqref{lp}.

\begin{theorem} \label{converge-la-lp-v}
Suppose that $\{\epsilon^k\}$ is a sequence of non-increasing positive vectors in $\Re^n$ 
and $\epsilon^k \to 0$ as $k \to \infty$. Let the sequence $\{x^k\}$ be generated by the above 
$\irl_\alpha$ method. Suppose that $x^*$ is an accumulation point of $\{x^k\}$. Then, $x^*$ 
is a stationary point of \eqref{lp}, i.e., \eqref{1st-cond} holds at $x^*$.
\end{theorem}

\begin{proof}
Since $\bF_\ae(x) \ge F(x)  \ge \underline f$ for every $x\in \Re^n$, we see
that $\{\bF_{\ae^k}(x^k)\}$ is bounded below. In addition, $\{\bF_{\ae^k}(x^k)\}$ is 
non-increasing due to \eqref{descent-gap-la}. Hence, $\{\bF_{\ae^k}(x^k)\}$ converges, 
which together with \eqref{descent-gap-la} implies that $\|x^{k+1}-x^k\| \to 0$. 
Let $\bar L_k$ denote the final value of $L_k$ at the $k$th outer iteration. By a 
similar argument as in the proof of Theorem \ref{valpha-outer}, we can show 
that $\bar L_k \in [L_{\min}, \tau(L_f+c)/2)$. Let $\cB = \{i|x^*_i \neq 0\}$.
Since $x^*$ is an accumulation point of $\{x^k\}$, there exists a subsequence $K$ 
such that $\{x^k\}_K \to x^*$. By the definition of $s^k$, we see that 
$\lim\limits_{k\in K \to \infty} s^k_i = |x^*_i|^{p-\alpha}$. The first-order optimality 
condition of \eqref{la-close-subprob} with $L_k = \bar L_k$ yields
\[
\frac{\partial f(x^{k+1})}{\partial x_i} +\bar L_k (x^{k+1}_i-x^k_i) + \lambda p 
s^k_i |x^{k+1}_i|^{\alpha-1}\sgn(x^{k+1}_i)  = 0,  \ \ \ \forall i \in \cB.
\] 
Upon taking limits on both sides of the above equality as $k\in K \to \infty$, and using 
the relation $\lim\limits_{k\in K \to \infty} s^k_i = |x^*_i|^{p-\alpha}$, we have
\[
\frac{\partial f(x^*)}{\partial x_i} + \lambda p |x^*_i|^{p-1}\sgn(x^*_i)  = 0,  
\ \ \ \forall i \in \cB.
\]
Using this relation and a similar argument as in the proof of Theorem \ref{nec-conds} 
(i), we can conclude that $x^*$ satisfies \eqref{1st-cond}.
\end{proof}

\section{New iterative reweighted $l_1$ minimization for \eqref{lp}}
\label{new-IRL1}

The $\irl_1$ and $\irl_2$ methods studied in Section \ref{unify} 
require that the parameter $\eps$ be dynamically adjusted and approach zero. 
One natural question is whether an iterative reweighted minimization 
method can be proposed for \eqref{lp} that shares a similar convergence 
with those methods but does not need to adjust $\eps$. We will address 
this question by proposing a new $\irl_1$ method and its variant.

As shown in Subsection \ref{lip-approx}, problem \eqref{lp-eps} has a locally Lipschitz 
continuous objective function and it is an $\eps$-approximation to \eqref{lp}. Moreover, 
when $\eps$ is below a computable threshold value, a certain stationary point of \eqref{lp-eps} 
is also that of \eqref{lp}. In this section we propose new $\irl_1$ methods for 
solving \eqref{lp}, which can be viewed as the $\irl_1$ methods directly applied to problem 
\eqref{lp-eps}. The novelty of these methods is in that the parameter $\eps$ is chosen only 
once and then fixed throughout all iterations. Remarkably, we are able to establish that any 
accumulation point of the sequence generated by these methods is a first-order stationary 
point of \eqref{lp}.
%In contrast to our methods, the existing iterative reweighted minimization methods in the 
%literature \cite{} need to gradually reduce $\eps$ by forcing it to go to zero.   

\gap

\noindent
%\begin{minipage}[h]{6.6 in}
{\bf New $\irl_1$ minimization method for \eqref{lp}:}  \\ [5pt]
Let $q$ be defined in \eqref{q}. Choose  an arbitrary $x^0 \in \Re^n$ and $\eps$ such that \eqref{eps-uppbdd} holds.  
Set $k=0$. 
\begin{itemize}
\item[1)] Solve the weighted $l_1$ minimization problem 
\beq \label{l1-subprob}
x^{k+1} \in \Arg\min \left\{f(x)+\lambda p \sum^n_{i=1} s^k_i |x_i|\right\},
\eeq
where $s^k_i = \min\left\{(\frac{\eps}{\lambda n})^{\frac1q},|x^k_i|^{\frac{1}{q-1}}\right\}$ 
for all $i$.
\item[2)]
Set $k \leftarrow k+1$ and go to step 1). 
\end{itemize}
\noindent
{\bf end}

\gap

We next establish that any accumulation point of $\{x^k\}$ generated by the above method 
is a first-order stationary point of \eqref{lp}.

\begin{theorem} \label{wl1-eps}
 Let the sequence $\{x^k\}$ be generated by the above $\irl_1$ method. Assume that $\eps$ 
satisfies \eqref{eps-uppbdd}. Suppose that $x^*$ is an accumulation point of $\{x^k\}$. Then 
$x^*$ is a first-order stationary point of \eqref{lp}, i.e., \eqref{1st-cond} holds at $x^*$. 
Moreover, the nonzero entries of $x^*$ satisfy the first-order bound \eqref{1st-bdd}. 
\end{theorem}

\begin{proof}
Let $\uel =(\frac{\eps}{\lambda n})^{1/q}$ and 
\beq \label{Gxs}
G(x,s) = f(x)+\lambda p \sum^n_{i=1} \left[|x_i|s_i- \frac{s^q_i}{q}\right].
\eeq 
By the definition of  $\{s^k\}$, one can observe that for $k \ge 0$,
\beq \label{bcd}
s^k = \arg\min\limits_{0 \le s \le \uel} G(x^k,s), \ \ \ 
x^{k+1} \in \Arg\min\limits_x G(x,s^k). 
\eeq
In addition, we observe that $F_\eps(x) = \min\limits_{0 \le s \le \uel} G(x,s)$ 
and $F_\eps(x^k)=G(x^k,s^k)$ for all $k$, where $F_\eps$ is defined in \eqref{feps}. It then follows that 
\beq \label{fval}
F_\eps(x^{k+1}) \ = \ G(x^{k+1},s^{k+1}) \ \le \ G(x^{k+1},s^k) \ \le \ G(x^k,s^k) \ = \ F_\eps(x^k).
\eeq
Hence, $\{F_\eps(x^k)\}$ is non-increasing. Since $x^*$ is an accumulation point of $\{x^k\}$, there 
exists a subsequence $K$ such that $\{x^k\}_{K} \to x^*$. By the continuity of $F_\eps$, we have 
$\{F_\eps(x^k)\}_K \to F_\eps(x^*)$, which together with the monotonicity of $\{F_\eps(x^k)\}$ 
implies that $F_\eps(x^k) \to F_\eps(x^*)$. Let $s^*_i = \min\{u_\eps,|x^*_i|^{\frac{1}{q-1}}\}$ for all $i$. 
We then observe that $\{s^k\}_K \to s^*$ and $F_\eps(x^*)=G(x^*,s^*)$. Using \eqref{fval} and 
$F_\eps(x^k) \to F_\eps(x^*)$, we see that $G(x^{k+1},s^k) \to F_\eps(x^*)=G(x^*,s^*)$. In addition, 
it follows from \eqref{bcd} that 
\[
G(x,s^k) \ \ge \ G(x^{k+1},s^k) \ \ \ \forall x\in \Re^n.
\]
Upon taking limits on both sides of this inequality as $k\in K \to \infty$, we have 
\[
G(x,s^*) \ \ge \ G(x^*,s^*) \ \ \ \forall x\in \Re^n,
\]
that is, 
\beq \label{lim-pt}
x^* \in  \Arg\min \left\{f(x)+\lambda p \sum^n_{i=1} s^*_i |x_i|\right\}.
\eeq
%Let $\cB = \{i|x^*_i \neq 0\}$. It then follows from \eqref{lim-pt} that 
%\beq \label{sub-opt}
%x^* \in \Arg\min\limits_{x\in\Re^n} \left\{f(x)+\lambda p \sum\limits_{i\in \cB} s^*_i |x_i|: 
%x_i = 0, \ i \notin \cB\right\}. 
%\eeq
The first-order optimality condition of \eqref{lim-pt} yields
\beq \label{1st-lim-opt}
0 \in \frac{\partial f(x^*)}{\partial x_i} + \lambda p s^*_i \sgn(x^*_i) ,  
\ \ \ \forall i.
\eeq  
Recall that $s^*_i = \min\{u_\eps,|x^*_i|^{\frac{1}{q-1}}\}$, which together with \eqref{q} 
implies that for all $i$, 
\[
s^*_i = \left\{\ba{ll}
|x^*_i|^{p-1}, & \mbox{if} \ |x^*_i| > u_\eps^{q-1}, \\
u_\eps, & \mbox{if} \ |x^*_i| \le u_\eps^{q-1}.
\ea\right.
\]
Substituting it into \eqref{1st-lim-opt} and using \eqref{deriv}, we obtain that 
\[
0 \in \frac{\partial f(x^*)}{\partial x_i} + \lambda \partial h_{u_\eps}(x^*_i),  
\ \ \ \forall i.
\] 
It then follows from \eqref{feps} that $x^*$ is a first-order stationary point of $F_\eps$. 
In addition, by the monotonicity of $\{F_\eps(x^k)\}$ and $F_\eps(x^k) \to F_\eps(x^*)$, 
we know that $F_\eps(x^*) \le F_\eps(x^0)$. Using these results and Theorem \ref{stat-pts}, 
we conclude that $x^*$ is  a first-order stationary point of \eqref{lp}. The rest of conclusion 
immediately follows from Theorem \ref{1st-lower-bdd}. 
\end{proof}

\gap

The above $\irl_1$ method needs to solve a sequence of reweighted $l_1$ minimization 
problems \eqref{l1-subprob} whose solution may not be cheaply computable. 
We next propose a variant of this method in which each subproblem has a 
closed form solution.

\gap

\noindent
{\bf A variant of new $\irl_1$ minimization method 
for \eqref{lp}:}  \\ [5pt]
Let $0< L_{\min} < L_{\max}$, $\tau>1$ and $c>0$ be given. Let $q$ be defined in \eqref{q}. 
Choose an arbitrary $x^0$ and $\eps$ such that \eqref{eps-uppbdd} holds. Set $k=0$. 
\begin{itemize}
\item[1)] Choose $L^0_k \in [L_{\min}, L_{\max}]$ arbitrarily. Set $L_k = L^0_k$.  
\bi
\item[1a)] Solve the weighted $l_1$ minimization problem 
\beq \label{l1-close-subprob}
x^{k+1} \in \Arg\min\limits_x \left\{f(x^k)+\nabla f(x^k)^T(x-x^k) + 
\frac{L_k}{2}\|x-x^k\|^2_2 + \lambda p \sum^n_{i=1} s^k_i |x_i|\right\},
\eeq
where $s^k_i = \min\left\{(\frac{\eps}{\lambda n})^{\frac1q},
|x^k_i|^{\frac{1}{q-1}}\right\}$ for all $i$. 
\item[1b)] If 
\beq \label{descent-gap}
F_\eps(x^k) - F_\eps(x^{k+1}) \ge \frac{c}{2} \|x^{k+1}-x^k\|^2_2
\eeq 
is satisfied, where $F_\eps$ is defined in \eqref{feps}, then go to step 2). 
\item[1c)] Set $L_k \leftarrow \tau L_k$ and go to step 1a).
\ei
\item[2)]
Set $k \leftarrow k+1$ and go to step 1). 
\end{itemize}
\noindent
{\bf end}

\vgap

We first show that for each outer iteration, the number of its inner 
iterations is finite.  

\begin{theorem} \label{inner-iteration}
For each $k \ge 0$, the inner termination criterion \eqref{descent-gap} 
is satisfied after at most $\left\lceil \frac{\log(L_f+c)-\log(2L_{\min})}
{\log \tau} +2\right\rceil$ inner iterations.
\end{theorem}

\begin{proof}
Let $\bar L_k$ denote the final value of $L_k$ at the $k$th outer 
iteration. Since the objective function of \eqref{l1-close-subprob} is strongly convex with 
modulus $L_k$, we have 
\[
\ba{lcl}
f(x^k)+\lambda p \sum^n_{i=1} s^k_i |x^k_i| &\ge & f(x^k)+\nabla f(x^k)^T(x^{k+1}-x^k) 
+ \lambda p \sum^n_{i=1} s^k_i |x^{k+1}_i| + L_k\|x^{k+1}-x^k\|^2_2, \\ [6pt]
& \ge & f(x^{k+1}) + 
\lambda p \sum^n_{i=1} s^k_i |x^{k+1}_i| + (L_k-\frac{L_f}{2})\|x^{k+1}-x^k\|^2_2, 
\ea
\]
where the last inequality is due to \eqref{lip-ineq}. This inequality together with 
\eqref{Gxs} yields 
\[
G(x^k,s^k) \ \ge \ G(x^{k+1},s^k) + (L_k-\frac{L_f}{2})\|x^{k+1}-x^k\|^2_2.
\]
Recall that $F_\eps(x) = \min\limits_{0 \le s \le \uel} G(x,s)$ and 
$F_\eps(x^k)=G(x^k,s^k)$, where $\uel=(\frac{\eps}{\lambda n})^{1/q}$. 
Using these relations and the above inequality, we obtain that
\[
\ba{lcl}
F_\eps(x^{k+1}) &=& G(x^{k+1},s^{k+1}) \ \le \ G(x^{k+1},s^k) \ \le \ 
G(x^k,s^k) -(L_k-\frac{L_f}{2})\|x^{k+1}-x^k\|^2_2 \\ [6pt]
&=&  F_\eps(x^k) - (L_k-\frac{L_f}{2})\|x^{k+1}-x^k\|^2_2.
\ea
\]
Hence, \eqref{descent-gap} holds whenever $L_k \ge (L_f+c)/2$. The rest of the proof is 
similar to that of Theorem \ref{inner-iteration}. 
\end{proof}

\gap

We next establish that any accumulation point of the sequence $\{x^k\}$ generated above 
is a first-order stationary point of problem \eqref{lp}.

\begin{theorem} \label{wl1-close-eps}
Let the sequence $\{x^k\}$ be generated by the above variant of new $\irl_1$ method. Assume 
that $\eps$ satisfies \eqref{eps-uppbdd}. Suppose that $x^*$ is an accumulation point of $\{x^k\}$. 
Then $x^*$ is a first-order stationary point of \eqref{lp}, i.e., \eqref{1st-cond} holds at $x^*$. 
Moreover, the nonzero entries of $x^*$ satisfy the first-order bound \eqref{1st-bdd}. 
\end{theorem}

\begin{proof}
It follows from \eqref{descent-gap} that $\{F_\eps(x^k)\}$ is non-increasing.
Since $x^*$ is an accumulation point of $\{x^k\}$, there exists a subsequence $K$ such 
that $\{x^k\}_{K} \to x^*$. By the continuity of $F_\eps$, we have 
$\{F_\eps(x^k)\}_K \to F_\eps(x^*)$, which together with the monotonicity of $\{F_\eps(x^k)\}$ 
implies that $F_\eps(x^k) \to F_\eps(x^*)$. Using this result and \eqref{descent-gap}, we 
can conclude that $\|x^{k+1}-x^k\| \to 0$. Let $\bar L_k$ denote the final value of $L_k$ at 
the $k$th outer iteration. By a similar argument as in the proof of Theorem \ref{valpha-outer}, 
one can show that $\bar L_k \in [L_{\min}, \tau(L_f+c)/2)$. The first-order optimality condition 
of \eqref{l1-close-subprob} with $L_k = \bar L_k$ yields
\[
0 \in \frac{\partial f(x^k)}{\partial x_i} + \bar L_k (x^{k+1}_i-x^k_i) + \lambda p 
s^k_i \sgn(x^{k+1}_i)  = 0,  \ \ \ \forall i.
\] 
Upon taking limits on both sides of the above equality as $k\in K \to \infty$, we 
have
\[
0 \in \frac{\partial f(x^*)}{\partial x_i} + \lambda p s^*_i \sgn(x^*_i),  
\ \ \ \forall i,
\]
where $s^*_i = \min\{(\frac{\eps}{\lambda n})^{1/q},|x^*_i|^{\frac{1}{q-1}}\}$ for all $i$. The rest of 
the proof is similar to that of Theorem \ref{wl1-eps}.
\end{proof}

\section{Computational results}
\label{results}

In this section we conduct numerical experiment to compare the performance of the variants of $\irl_1$ 
methods proposed in Subsection \ref{1st-type} and \ref{2nd-type} and Section \ref{new-IRL1}. In particular, 
we apply these methods to problem \eqref{l2-lp} whose data are randomly generated. For convenience of 
presentation, we name these variants as $\irl_1$-1, $\irl_1$-2 and $\irl_1$-3, respectively. All codes 
are written in MATLAB and all computations are performed on a MacBook Pro running with Mac OS X Lion 10.7.4 
and 4GB memory.

For all three methods, we choose $L_{\min}=1\rm e$-8, $L_{\max}=1\rm e$+8, $c=1\rm e$-4, $\tau=1.1$, and 
$L^0_0=1$. And we update $L^0_k$ by the similar strategy as used in spectral projected gradient method 
\cite{BiMaRa00}, that is, 
\[
L^0_k = \max\left\{L_{\min},\min\left\{L_{\max},\frac{\Delta x^T \Delta g}{\|\Delta x\|^2}\right\}\right\},
\]
where $\Delta x = x^k -x^{k-1}$ and $\Delta g = \nabla f(x^k) - \nabla f(x^{k-1})$. In addition, 
we choose $\eps^k=0.1^k e$ and $\delta_k = 0.1^k$ for $\irl_1$-1 and 
$\eps^k = 0.995^k e$ for $\irl_1$-2, respectively, where $e$ is the all-ones vector. 
For $\irl_1$-3, $\eps$ is chosen to be the one satisfying \eqref{eps-uppbdd} but within 
$10^{-6}$ to the supremum of all $\eps$'s satisfying \eqref{eps-uppbdd}.  The same initial 
point $x^0$ is used for $\irl_1$-1, $\irl_1$-2 and $\irl_1$-3. In particular,  we choose $x^0$ 
to be 
\[
x^0 \in \Arg\min\left\{\frac12\|Ax-b\|^2+\lambda\|x\|_1\right\},
\]
which can be computed by a variety of methods (e.g., \cite{VaFr08,BeTe09,HaYiZh07,WrNoFi09,YuTo11}).
And all methods terminate according to the following criterion
\[
\|X \nabla f(x) + \lambda p |x|^p\|_{\infty} \le \rm{1e-4},
\]
where $X=\Diag(x)$. 

In the first experiment, we set $\lambda = 3{\rm e}$-3 for problem \eqnok{l2-lp}. And the 
data $A$ and $b$ are randomly generated in the same manner as described in $l_1$-magic \cite{CanRom05}. 
In particular, given $\sigma >0$ and positive integers $m$, $n$, $T$ with 
$m < n$ and $T < n$, we first generate a matrix $W\in\Re^{n\times m}$ with entries 
randomly chosen from a normal distribution with mean zero, variance one and standard deviation 
one. Then we compute an orthonormal basis, denoted by $B$, for the range space 
of $W$, and set $A=B^T$. We also randomly generate a vector $\tx \in \Re^n$ with 
only $T$ nonzero components that are $\pm 1$, and generate a vector $v\in\Re^{m}$ 
with entries randomly chosen from a normal distribution with mean zero, variance one 
and standard deviation one. Finally, we set $b = A\tx+\sigma v$. Especially, we choose 
$\sigma=0.005$ for all instances. 

 The results of these methods for the above randomly generated instances 
with $p=0.1$ and $0.5$ are presented in Tables \ref{res1-1} and \ref{res1-2}, respectively. In 
detail, the parameters $m$ and $n$ of each instance are listed in the first two columns, respectively. 
The objective function value of problem \eqref{l2-lp}  for these methods is given in columns three to five, 
and CPU times (in seconds) are given in the last three columns, respectively. We shall mention that the 
CPU time reported here does not include the time for obtaining initial point $x^0$.  For $p=0.1$, we observe 
from Table \ref{res1-1} that all three methods produce similar objective function values. The CPU time of 
$\irl_1$-1 and $\irl_1$-3 is very close, which is much less than that of $\irl_1$-2. 
For $p=0.5$, we see from Table \ref{res1-2} that $\irl_1$-1 and $\irl_1$-3 achieve better objective function 
values than $\irl_1$-2 while the former two methods require less CPU time.

\begin{table}[t!]
\caption{Comparison of three $\irl_1$ methods for problem \eqref{l2-lp} with $p=0.1$}
\centering
\label{res1-1}
%\begin{center}
%\begin{small}
\begin{tabular}{|rr||rrr||rrr|}
\hline 
\multicolumn{2}{|c||}{Problem} & \multicolumn{3}{c||}{Objective Value} &
 \multicolumn{3}{c|}{CPU Time} \\ 
\multicolumn{1}{|c}{m} & \multicolumn{1}{c||}{n}  
& \multicolumn{1}{c}{\sc $\irl_1$-1} & \multicolumn{1}{c}{\sc $\irl_1$-2} & 
\multicolumn{1}{c||}{\sc $\irl_1$-3} &  \multicolumn{1}{c}{\sc $\irl_1$-1} & 
\multicolumn{1}{c}{\sc $\irl_1$-2} & \multicolumn{1}{c|}{\sc $\irl_1$-3}\\
\hline
120 & 512 &  0.061371 & 0.061371 & 0.061371 & 0.02 & 0.29 & 0.01\\
240 & 1024 &  0.122579 & 0.122579 & 0.122579 & 0.01 & 0.47 & 0.01\\
360 & 1536 &  0.183595 & 0.183595 & 0.183595 & 0.01 & 0.77 & 0.01\\ 
480 & 2048 &  0.245253 & 0.245253 & 0.245253 & 0.02 & 1.45 & 0.02\\
600 & 2560 &  0.305575 & 0.305575 & 0.305575 & 0.03 & 2.30 & 0.03\\
720 & 3072 & 0.367497 & 0.367496 & 0.347697 & 0.04 & 3.11 & 0.04\\
840 & 3584 &  0.429549 & 0.429548 & 0.429549 & 0.05 & 3.83 & 0.06\\
960 & 4096 &  0.489512 & 0.489512 & 0.489512 & 0.06 & 5.32 & 0.08\\
1080 & 4608 &  0.550911 & 0.550911 & 0.554911 & 0.07 & 6.59 & 0.10\\
1200 & 5120 &  0.611896 & 0.611896 & 0.611896 & 0.10 & 7.51 & 0.13\\
\hline
\end{tabular}
%\end{small}
%\end{center}
\end{table}

\begin{table}[t!]
\caption{Comparison of three $\irl_1$ methods for problem \eqref{l2-lp} with $p=0.5$}
\centering
\label{res1-2}
%\begin{center}
%\begin{small}
\begin{tabular}{|rr||rrr||rrr|}
\hline 
\multicolumn{2}{|c||}{Problem} & \multicolumn{3}{c||}{Objective Value} &
 \multicolumn{3}{c|}{CPU Time} \\ 
\multicolumn{1}{|c}{m} & \multicolumn{1}{c||}{n}  
& \multicolumn{1}{c}{\sc $\irl_1$-1} & \multicolumn{1}{c}{\sc $\irl_1$-2} & 
\multicolumn{1}{c||}{\sc $\irl_1$-3} &  \multicolumn{1}{c}{\sc $\irl_1$-1} & 
\multicolumn{1}{c}{\sc $\irl_1$-2} & \multicolumn{1}{c|}{\sc $\irl_1$-3}\\
\hline
120 & 512 &  0.061298 & 0.062003 & 0.061298 & 0.02 & 0.17 & 0.01\\
240 & 1024 &  0.122412 & 0.123449 & 0.122412 & 0.01 & 0.26 & 0.01\\
360 & 1536 &  0.183376 & 0.184881 & 0.183376 & 0.01 & 0.43 & 0.01\\
480 & 2048 & 0.244745 & 0.247495 & 0.244745 & 0.02 & 0.90 & 0.02\\
600 & 2560 &  0.304945 & 0.306632 & 0.304945 & 0.03 & 1.55 & 0.03\\
720 & 3072 &  0.366621 & 0.370576 & 0.366621 & 0.03 & 2.07 & 0.04\\
840 & 3584 &  0.429043 & 0.433426 & 0.429043 & 0.04 & 2.57 & 0.06\\
960 & 4096 & 0.488704 & 0.492537 & 0.488704 & 0.05 & 3.54 & 0.08\\
1080 & 4608 &  0.550031 & 0.554057 & 0.550031 & 0.06 & 4.40 & 0.10\\
1200 & 5120 &  0.610850 & 0.615399 & 0.610850 & 0.07 & 5.26 & 0.12\\
\hline
\end{tabular}
%\end{small}
%\end{center}
\end{table}

In the second experiment, we also randomly generate all instances for problem 
\eqnok{l2-lp}. In particular, we generate matrix $A$ and vector $b$ with entries 
randomly chosen from standard uniform distribution.  In addition, we set
 $\lambda = 3{\rm e}$-3 for \eqnok{l2-lp}. The results of these methods for 
the above randomly generated instances with $p=0.1$ and $0.5$ are presented 
in Tables \ref{res2-1} and \ref{res2-2}, respectively. Same as above, the CPU time reported 
here does not include the time for obtaining initial point $x^0$. For $p=0.1$, we observe 
from Table \ref{res2-1} that among $\irl_1$-1, $\irl_1$-2 and $\irl_1$-3 achieves best 
objective function values over 4, 4 and 3 instances out of total 10 instances, respectively. 
The average CPU time of $\irl_1$-3 is much less than that of $\irl_1$-1 and $\irl_1$-2. 
For $p=0.5$, all three methods achieve similar objective function values. The overall CPU 
time of $\irl_1$-2 and $\irl_1$-3 is very close, which is much less than that of $\irl_1$-1. 

From the above two experiments, we observe that $\irl_1$-3 is generally more stable 
than $\irl_1$-1 and $\irl_1$-2 in terms of objective function value and CPU time. 

\begin{table}[t!]
\caption{Comparison of three $\irl_1$ methods for problem \eqref{l2-lp} with $p=0.1$}
\centering
\label{res2-1}
%\begin{center}
%\begin{small}
\begin{tabular}{|rr||rrr||rrr|}
\hline 
\multicolumn{2}{|c||}{Problem} & \multicolumn{3}{c||}{Objective Value} &
 \multicolumn{3}{c|}{CPU Time} \\ 
\multicolumn{1}{|c}{m} & \multicolumn{1}{c||}{n}  
& \multicolumn{1}{c}{\sc $\irl_1$-1} & \multicolumn{1}{c}{\sc $\irl_1$-2} & 
\multicolumn{1}{c||}{\sc $\irl_1$-3} &  \multicolumn{1}{c}{\sc $\irl_1$-1} & 
\multicolumn{1}{c}{\sc $\irl_1$-2} & \multicolumn{1}{c|}{\sc $\irl_1$-3}\\
\hline
120 & 512 & 0.6557 & 0.6007 & 0.6011 & 1.25 & 1.86 & 0.93 \\
240 & 1024 &  1.1916 & 1.2090 & 1.2108 & 1.96 & 3.99 & 2.16\\
360 & 1536 &  1.7047 & 1.6955 & 1.7253 & 3.46 & 7.34 & 2.74\\
480 & 2048 &  2.3025 & 2.3112 & 2.3270 & 9.68 & 14.91 & 7.86\\
600 & 2560 &  2.7888 & 2.7432 & 2.7432 & 13.50 & 27.29 & 20.90\\
720 & 3072 &  3.3639 & 3.4051 & 3.4296 & 19.96 & 36.75 & 21.51\\
840 & 3584 &  3.7613 & 3.7614 & 3.7085 & 24.26 & 46.68 & 36.10\\
960 & 4096 & 4.4721 & 4.2879 & 4.2980 & 60.26 & 58.77 & 47.98\\
1080 & 4608 & 5.0258 & 4.8848 & 4.8649 & 72.45 & 69.51 & 39.35 \\
1200 & 5120 &  5.2228 & 5.3789 & 5.3561 & 83.99 & 91.26 & 57.97\\
\hline
\end{tabular}
%\end{small}
%\end{center}
\end{table}

\begin{table}[t!]
\caption{Comparison of three $\irl_1$ methods for problem \eqref{l2-lp} with $p=0.5$}
\centering
\label{res2-2}
%\begin{center}
%\begin{small}
\begin{tabular}{|rr||rrr||rrr|}
\hline 
\multicolumn{2}{|c||}{Problem} & \multicolumn{3}{c||}{Objective Value} &
 \multicolumn{3}{c|}{CPU Time} \\ 
\multicolumn{1}{|c}{m} & \multicolumn{1}{c||}{n}  
& \multicolumn{1}{c}{\sc $\irl_1$-1} & \multicolumn{1}{c}{\sc $\irl_1$-2} & 
\multicolumn{1}{c||}{\sc $\irl_1$-3} &  \multicolumn{1}{c}{\sc $\irl_1$-1} & 
\multicolumn{1}{c}{\sc $\irl_1$-2} & \multicolumn{1}{c|}{\sc $\irl_1$-3}\\
\hline
120 & 512 &  0.2408 & 0.2415 & 0.2405 & 2.12 & 1.57 & 0.99\\
240 & 1024 &  0.4096 & 0.4127 & 0.4140 & 7.10 & 3.12 & 2.42\\
360 & 1536 &  0.5361 & 0.5336 & 0.5336 & 21.50 & 5.31 & 4.04\\
480 & 2048 & 0.6900 & 0.6900 & 0.6934 & 34.93 & 14.07 & 9.95\\
600 & 2560 &  0.7725 & 0.7772 & 0.7739 & 61.08 & 21.68 & 25.49\\
720 & 3072 &  0.9393 & 0.9405 & 0.9406 & 259.72 & 34.94 & 35.55\\
840 & 3584 &  1.0113 & 1.007 & 1.007 & 313.30 & 47.24 & 39.43\\
960 & 4096 & 1.1403 & 1.1297 & 1.1280 & 533.36 & 54.50 & 52.90\\
1080 & 4608 &  1.2178 & 1.2186 & 1.2220 & 348.94 & 77.42 & 80.55\\
1200 & 5120 &  1.3291 & 1.3375 & 1.3375 & 835.89 & 104.27 & 114.99\\
\hline
\end{tabular}
%\end{small}
%\end{center}
\end{table}

\section{Concluding remarks}
\label{conclude}

In this paper we studied iterative reweighted minimization methods for $l_p$ regularized 
unconstrained minimization problems \eqref{lp}. In particular, we derived lower 
bounds for nonzero entries of first- and second-order stationary points, and hence also of 
local minimizers of \eqref{lp}. We extended some existing $\irl_1$ and $\irl_2$ methods to 
solve \eqref{lp} and proposed new variants for them. Also, we provided a unified convergence analysis 
for these methods. In addition, we proposed a novel Lipschitz continuous $\eps$-approximation 
to $\|x\|^p_p$. Using this result, we developed new $\irl_1$ methods for \eqref{lp} and showed 
that any accumulation point of the sequence generated by these methods is a first-order 
stationary point of problem \eqref{lp}, provided that the approximation parameter $\eps$ is 
below a computable threshold value. This is a remarkable result since all existing iterative 
reweighted minimization methods require that $\eps$ be dynamically updated and approach 
zero. Our computational results demonstrate that the new $\irl_1$ method  is generally more
stable than the existing $\irl_1$ methods \cite{FoLa09,ChZh10} in terms of objective function 
value and CPU time. 

Recently, Zhao and Li \cite{ZhLi12} proposed an $\irl_1$ minimization method to identify
sparse solutions to undetermined linear systems based on a class of regularizers. When applied 
to the $l_p$ regularizer, their method becomes one of the first type of $\irl_1$ methods discussed 
in Subsection \ref{1st-type}. Though we only studied the $l_p$ regularized minimization problems, 
the techniques developed in our paper can be useful for analyzing the iterative reweighted minimization 
methods for the optimization problem with other regularizers.

\end{document}